\documentclass{article}
\usepackage{amsmath,amsfonts,amssymb,mathrsfs,mathtools,url}

\usepackage[margin=26mm]{geometry}

\title{%
  Asymptotic Improvements on the Exact Matching Distance
       \newline for $2$-parameter Persistence%
  \thanks{This work is supported by the Austrian Science Fund (FWF) grant number P 33765-N.}
}

\author{%
  H\aa vard Bakke Bjerkevik,%
  \thanks{Graz University of Technology; \emph{currently at} SUNY Albany, 
          \texttt{hbjerkevik@albany.edu}}\,
  Michael Kerber,%
  \thanks{Graz University of Technology
          \texttt{kerber@tugraz.at}}
}

\usepackage{amsthm}
\theoremstyle{plain}
\newtheorem{theorem}{Theorem}
\newtheorem{lemma}[theorem]{Lemma}
\newtheorem{proposition}[theorem]{Proposition}
\newtheorem{corollary}[theorem]{Corollary}
\newtheorem{definition}[theorem]{Definition}

\usepackage[capitalise]{cleveref}
\usepackage{tikz}

\usetikzlibrary{matrix, arrows, calc, patterns, intersections, decorations.markings, shapes.geometric, chains, decorations.pathmorphing}

\DeclareMathOperator{\push}{push}

\DeclareMathOperator{\gr}{gr}
\DeclareMathOperator{\spn}{span}

\newcommand{\change}[1]{#1}

\newcommand{\I}{\mathbb I}
\newcommand{\R}{\mathbb R}

\newcommand{\bigO}{O}

\newcommand{\Ll}{\mathcal{L}}
\newcommand{\G}{\mathcal{G}}
\newcommand{\Rel}{\mathcal{R}}
\newcommand{\M}{\mathcal{M}}
\newcommand{\T}{\mathcal{T}}

\newcommand{\Hh}{\mathcal{H}}
\newcommand{\Ss}{\mathcal{S}}
\newcommand{\Pp}{\mathcal{P}}
\newcommand{\X}{\mathcal{X}}
\newcommand{\B}{\mathcal{B}}
\newcommand{\UT}{\bigcup \T}

\begin{document}
\maketitle

\begin{abstract}
In the field of topological data analysis, persistence modules are used to express geometrical features of data sets.
The matching distance $d_\M$ measures the difference between $2$-parameter persistence modules by taking the maximum bottleneck distance between $1$-parameter slices of the modules.
The previous best algorithm to compute $d_\M$ exactly runs in $O(n^{8+\omega})$ time using $O(n^4)$ space, where $n$ is the number of generators and relations of the modules and $\omega$ is the matrix multiplication constant.
We improve significantly on this by describing an algorithm with expected running time $O(n^5 \log^3 n)$ and using $O(n^2)$ space.
We first solve the decision problem $d_\M\leq \lambda$ for a constant $\lambda$ in $O(n^5\log n)$ time by traversing a line arrangement in the dual plane, where each point represents a slice.
Then we lift the line arrangement to a plane arrangement in $\R^3$ whose vertices represent possible values for $d_\M$, and use a randomized incremental method to search through the vertices and find $d_\M$.
The expected running time of this algorithm is $\bigO((n^4+T(n))\log^2 n)$, where $T(n)$ is an upper bound for the complexity of deciding if $d_\M\leq \lambda$.
Moreover, we show how to compute the matching distance using only linear space, to the price of a much worse time complexity.
\end{abstract}

\section{Introduction}
\emph{Persistent homology}~\cite{eh-computational,elz-topological,oudot-book} is a method to summarize the topological
properties of a data set across different scales. The last 20 years
have witnessed numerous applications of persistent homology in various
scientific fields, including neuroscience~\cite{bmmps-persistent,reimann-cliques,rbd-decoding}, material science~\cite{hiraoka-hierarchical,lee-high}, cosmology~\cite{cosmo2,cosmo1}, and many others,
shaping the field of \emph{topological data analysis}~\cite{carlsson-survey}. Part of the success
of this field is the presence of fast algorithms for the most important
sub-tasks, namely computing a discrete summary of the topology of the data
(the so-called \emph{barcode}) via Gaussian elimination of a sparse matrix,
and computing the distance of two such barcodes (the so-called \emph{bottleneck distance}) as a minimal cost matching via the Hopcroft-Karp algorithm.

The standard theory of persistent homology assumes a single scalar (real)
parameter to determine the scale. In many applications, however, it is natural
to alter more than one parameter: As a simple example, consider the
situation of a point cloud in Euclidean space, where one parameter controls
until what distance two points are considered as close, and a second
parameter determines until which local density a point is labeled as an outlier.
Every combination of these two parameters results in a shape approximating
the data, and studying the evolution of topological properties across the
$2$-parameter space of shapes leads to the theory of \emph{multi-parameter persistence}.
Unfortunately, already for two parameters, 
there is no direct generalization of the barcode~\cite{cz-theory},
and the most natural notion of distance between two $2$-parameter data sets,
the \emph{interleaving distance}, is NP-hard to approximate
for any factor smaller than $3$~\cite{bbk-computing}.

Despite these negative results, there are alternative notions of distance
which allow for a polynomial-time computation. Among them is the \emph{matching distance} ($d_\M$), the topic of this paper. The main idea is that when imposing 
a linear condition on the parameters, we effectively restrict the data
to one-parameter and can compute barcodes and bottleneck distance for this
restriction. The matching distance is then defined as the supremum of the bottleneck distance over all possible linear restrictions.

Kerber et al. showed recently~\cite{klo-exact} 
that the matching distance can be computed (exactly)
in polynomial time, more precisely, with an asymptotic time complexity of $O(n^{8+\omega})$, where $n$ is the size of the input and $\omega$ is the matrix
multiplication constant.

\paragraph{Contribution.}
We describe a randomized algorithm that computes the exact matching distance
in $\tilde{O}(n^5)$ time, where $\tilde{O}$ means that we ignore logarithmic factors in $n$. We achieve this substantial improvement in complexity through
a combination of several techniques from computational geometry,
computational topology, and graph matchings.
We only list the major ideas and refine them in Section~\ref{sec:detailed_overview} where we outline the approach in more detail.

\begin{itemize}
\item We focus on the decision question of whether $d_\M\leq\lambda$ or $d_\M>\lambda$ for a fixed $\lambda\in\R$. We show that this
question can be addressed using a line arrangement with $O(n^2)$ lines,
substantially reducing the size from the $O(n^4)$ lines needed
in the previous approach by Kerber et al.~\cite{klo-exact}.

\item We traverse the line arrangement and maintain the barcodes and the bottleneck matching between them during this traversal. Using the vineyard algorithm
by Cohen-Steiner et al.~\cite{cem-vines} and a result about computing bottleneck matchings from \cite{kmn-geometry}, this allows us to spend only
linear time per visited arrangement cell (up to logarithmic factors),
resulting in an $\tilde{O}(n^5)$ algorithm for the decision problem.

\item We identify a set of $O(n^6)$ candidate values for $\lambda$
that is guaranteed to contain the exact value of the matching distance.
These values are obtained from triple intersections in a plane arrangement with $O(n^2)$ planes.

\item We show how to avoid considering all candidate values through
randomized incremental construction (e.g.~\cite[Ch.~4]{dutch}). More precisely, we maintain upper and lower bounds on the matching distance and cluster the candidates
in groups. We can quickly rule out clusters if they do not improve
the bounds, and by randomization, we expect only a logarithmic number
of improvements of the bounds.
\end{itemize}

A careful realization of our algorithmic ideas yields a quadratic space complexity both for the decision problem
and the randomized incremental construction. In particular, our construction avoids to fully store
an arrangement of $O(n^2)$ lines in memory and improves upon the previous approach by Kerber et al.~\cite{klo-exact}
that had a space complexity of $O(n^4)$.
We also show that a further improvement in space complexity to $O(n)$ is possible that, however, 
has an exponential running time.

\paragraph{Motivation and Related work.}
A common theme in topological data analysis is to consider
topological properties of a data set as a proxy and analyze a collection of
data sets in terms of their topology. 
A distance measure allows for standard data analysis tasks
on such proxies, such as clustering or visualization via multi-dimensional
scaling. 

For such a distance measure to be meaningful, it is required that it is \emph{stable}, that is, a slightly
perturbed version of a data sets is in small distance to the unperturbed data
set. At the same time, the distance should also be \emph{discriminative}, that is, distinguish between data sets with fundamental differences. (In the extreme
case, the distance measure that assigns $0$ to all pairs is maximally stable,
but also maximally non-discriminative.) Finally, for practical purposes,
the distance, or at least an approximation of it, should also be efficiently computable.

For the case of one parameter, the bottleneck and Wasserstein distances
satisfy all these properties. They are both stable under
certain tameness conditions~\cite{ceh-stability,cehm-lipschitz,st-wasserstein}
and efficient software is available for computation~\cite{kmn-geometry}.
Moreover, the bottleneck distance is \emph{universal}, meaning that it
is the most discriminative distance among all stable distances~\cite{bl-induced}.

In the multi-parameter setup, the situation is more complicated:
the interleaving distance is universal~\cite{lesnick-univ-interl},
but NP-hard to compute~\cite{bbk-computing}. There is also a multi-parameter
extension of the bottleneck distance based on matching indecomposable elements,
which is unstable and only efficiently computable in special cases~\cite{dx-computing}. The matching distance is an attractive alternative, because
it is stable~\cite{landi2018rank} and computable in polynomial time.
Moreover, since it compares two multi-filtered data sets based on the
\emph{rank invariant}~\cite{cz-theory}, we can expect good discriminative
properties in practice, even if no theoretical statement can be made at this point. At least, first practical applications of the matching distance have
been identified in shape analysis~\cite{italian,cerri2013betti}
and computational chemistry~\cite{keller2018persistent}; see also~\cite{han2020distributions} for a recent case study.

The matching distance allows for a rather straightforward approximation
algorithm~\cite{italian} through adaptive geometric subdivision with a
quad-tree data structure.
Recently, this approach has been refined~\cite{kn-efficient}
and the code has been released in the \textsc{Hera} library\footnote{\url{https://github.com/grey-narn/hera}}. These developments raised the question of 
practical relevance for our approach, given that a good approximation of the distance is usually sufficient in practical data. 
We believe that, besides the general interest in the theoretical complexity
of the problem, an exact algorithm might turn out to be faster in practice
than the proposed approximation scheme. The reason is that to achieve
a satisfying approximation factor, the subdivision approach has to descend
quite deep in the quad-tree, and while the adaptiveness of quad-trees avoids
a full subdivision in some instances, there are others where large areas
of the domain have to be fully subdivided to certify the result~-- see for instance Figure 11 in~\cite{kn-efficient}. An exact approach might be able
to treat such areas as one component and save time. At the same time, 
we emphasize that the approach presented in this paper is optimized
for asymptotic complexity; we do not claim that it will be efficient 
in practice as described.

\section{Extended overview of our approach}
\label{sec:detailed_overview}

We give an overview describing all the main ideas of the paper, 
but skipping formal proofs of correctness and only introducing terms and technical constructions through simplified examples.
Everything needed to prove our results will be introduced independently in later sections.

\paragraph{Bottleneck distance.}
The input data in topological data analysis often takes the form of a filtration $(X_s)_{s\geq 0}$ of simplicial complexes. 
In this overview section, we denote by $n$ the number of simplices of the simplicial complex~-- the meaning of $n$ will be slightly different
in the technical part of the paper, but this change does not affect any of the stated complexity bounds.
Applying homology in some degree to the complexes and inclusion maps, one gets a $1$-parameter \emph{persistence module}, to which one can associate a barcode, which is a multiset of intervals (i.e., \emph{bars}) of the form $[b,d)$. Intuitively, each bar represents a homological feature that is born at $b$ and dies at $d$.
Given two such persistence modules, one would like to say something about how similar they are, which is usually done by comparing their barcodes.

We consider two barcodes $B_1$ and $B_2$ as close if $B_1$ can be transformed
into $B_2$ by changing all endpoints of bars in $B_1$ by a small value.
More concretely, consider a real value $\lambda\geq 0$ and consider all
barcodes that arise from $B_1$ by shifting both endpoints of every bar
by at most $\lambda$ to the left or right. In particular, a bar of length $\leq 2\lambda$ can be transformed to an empty interval, and thus removed. Likewise,
any bar of length $\leq 2\lambda$ can be created by conceptually inserting 
an empty interval $[c,c)$ in $B_1$, and shifting the left endpoint to the left
and the right endpoint to the right.
The \emph{bottleneck distance} of $B_1$ and $B_2$ is the smallest $\lambda$
such that $B_2$ can be obtained from $B_1$ by such a $\lambda$-perturbation.

The computation of the bottleneck distance can be reduced to computing 
maximum-cardinality matchings in bipartite graphs. The idea is to set up
a graph for a fixed $\lambda\geq 0$ with the bars of $B_1$ and $B_2$
as vertex set and connect two bars if they can be transformed into each other
via a $\lambda$-perturbation. Moreover, the graph has to be extended
to account for the possibility of removing and creating bars in the perturbation~-- see~\cite[Ch.~VIII.4]{eh-computational} for details. Setting up this graph
$G^\lambda$ properly, we get that the bottleneck distance is at most $\lambda$
if and only if $G^\lambda$ has a perfect matching. 

While the complexity for finding a maximum-cardinality matching is $O(n^{2.5})$ using the Hopcroft-Karp algorithm \cite{hopkarp}, it can be reduced to $O(n^{1.5}\log n)$
here because of the metric structure of the graph~\cite{eik-geometry}.
The idea is that a bar $[a,b)$ can be represented as a point $(a,b)$ in the plane (resulting in the so-called~\emph{persistence diagram}), and two points
are connected in $G^\lambda$ if their $\ell_\infty$-distance is at most $\lambda$. Then, in order to find an augmenting path, we can avoid traversing
all edges of $G^\lambda$ (whose size can be quadratic), but instead use 
a range tree data structure~\cite[Ch. 5]{dutch} to find neighbors
in the graph.

\paragraph{$2$-parameter filtrations.}

In some important cases, our filtration is equipped with two parameters instead of one; in this case one has a family $(X_{a,b})_{a\geq 0,b\geq 0}$ of simplicial complexes with inclusion maps $X_{a,b}\to X_{a',b'}$ when $a\leq a'$ and $b\leq b'$.
We assume that all complexes $X_{a,b}$ are contained
in a finite simplicial complex $X$ and 
we let $n$ denote the number of simplices of $X$.
As in the $1$-parameter setup, we are interested in the evolution
of the homology of $X_{a,b}$
as the parameters range over their domain.
There is no summary for that evolution in the same way as the barcode for
a single parameter. However, we can explore the $2$-parameter filtration
by fixing a line $s$ of positive slope, called a \emph{slice}, 
and consider the $1$-parameter family
of simplicial complexes $(X_{a,b})$, where $(a,b)$ ranges over all points on $s$.
One example is the family $(X_{a,a})_{a\geq 0}$, obtained from
the slice $x=y$. Because slices have positive slope, the induced $1$-parameter
family is a filtration in the above case, so there is a well-defined
barcode for the $2$-parameter filtration with respect to every slice $s$.
Moreover, having two $2$-parameter filtrations and a slice $s$, 
there is a well-defined bottleneck distance between the slice barcodes.\footnote{Technically, the barcode and bottleneck distance depend on the parameterization of the line, but we postpone the discussion to the technical part.}
The \emph{matching distance} between two $2$-parameter filtrations
is the supremum of the bottleneck distance over all slices.

These definitions carry over to an arbitrary number of parameters, but the
algorithm that follows will not, so we restrict to two parameters throughout.

\paragraph{Duality and bottleneck function\change{.}}
We rephrase the definition of the matching distance through the
well-known point-line duality of computational geometry~\cite[Ch.~8.2]{dutch}.
We assign for each line $y=ax+b$ with slope $a>0$ a \emph{dual point} $(a,b)\in \R_{>0}\times\R$.
A point $p=(p_x,p_y)\in \R^2$ dualizes to a line $y=-p_xx+p_y$, and the points on that dual
line are the duals of all slices that pass through $p$.
We refer to the parameter plane as the \emph{primal plane} and the space in which
the dual points and lines are contained as the \emph{dual plane} throughout.
For every dual point $(a,b)$ (with $a>0$), we can assign
a real value which is defined by the bottleneck distance of the two
slice barcodes with respect to the slice $y=ax+b$. This induces a function
$B:\R_{>0}\times\R\to\R$, which we call the \emph{bottleneck function}. 
The matching distance is then just the supremum of $B$ over its domain.

\paragraph{Barcode \change{pairings.}}
The advantage of the described duality is that the space of slices
can be decomposed into parts where the bottleneck function is
simpler to analyze. As a first step, we define the \emph{barcode pairing} 
of a one-parameter filtration of simplicial complexes
as the collection of pairs $(\sigma,\tau)$ where $\sigma$ is a simplex giving rise to a homology class
and $\tau$ is the simplex eliminating it. 
In the same way we obtained slice barcodes, we also have a slice barcode
pairing, and by duality, we can talk about the barcode pairing
of a point in the dual plane.

The barcode can be obtained by just replacing $\sigma$ and $\tau$ 
by the points at which they appear in the filtration, which we will refer to as their \emph{grades}.
Importantly, while the barcode depends on the grades,
the barcode pairing only depends on the \emph{order} of the grades.
Hence, when moving in the dual plane, the corresponding barcode pairings
remain the same as long as no two simplices switch order.

As observed by Lesnick and Wright~\cite{lw-interactive}, there is a line
arrangement decomposing the dual plane into ($2$-dimensional) regions,
such that the barcode pairing is identical for all dual points in a region.
\change{(The \emph{barcode templates} in \cite{lw-interactive} are the same as barcode pairings, except that they are defined in terms of grades instead of simplices.)}
The line arrangement is not difficult to define: let us consider
two simplices $\sigma$ with grade $(a,b)$ and $\sigma'$ with grade $(a',b')$ 
and assume $a<a'$. If also $b<b'$, $\sigma$ will precede $\sigma'$ in the filtration
of any slice. If $b>b'$, the simplices $\sigma$ and $\sigma'$ will have the same
critical value exactly for those slices that pass through the \emph{join}
of the two grades, which is the point $(a',b)$. Hence, the dual line $y=-a'x+b$ 
will split the dual plane into two half-planes where the order of the two 
simplices is fixed. Doing this for all pairs of simplices
yield the line arrangement which we call $\Ll$. It consists of $O(n^2)$ lines and therefore
has $O(n^4)$ regions.

\paragraph{Fast computation of barcode \change{pairings.}}
Naively, computing the barcode pairing for each region can be done in
$O(n^{4+\omega})$ time by computing the arrangement $\Ll$ ($O(n^4)$~\cite[Thm. 8.6]{dutch}), picking one dual point per region ($O(n^4)$), generating the filtration along the primal
slice ($O(n)$ per region), \change{and computing the slice barcodes ($O(n^\omega)$ per region).}

In \cite{lw-interactive}, it is shown
how to reduce the complexity to $O(n^5)$ using fast updates
of persistence diagrams under transpositions as explained in~\cite{cem-vines}.
Specifically, knowing the decomposition $RU=\partial$ for a fixed order
of simplices of a simplicial complex and a transposition of two neigboring simplices in that order, we can update the decomposition $(R,U)$ in linear 
time in the size of the matrix. Crossing a line of the arrangement $\Ll$ corresponds exactly
to such a transposition. This suggests the following approach:
Construct a spanning tree for the dual graph of the arrangement 
(``dual'' in a different sense than above: the vertices of the graph are 
the regions, and the edges correspond to pairs of regions that share an edge).
The spanning tree defines a walk through $\Ll$ with $O(n^4)$ steps
that visits every region at least once (by doubling every edge 
and using an Eulerian tour). We compute the barcode pairing 
for the starting region and keep applying 
the transposition algorithm of~\cite{cem-vines}
to get the barcode pairing for the next region in linear time.
This yields an $O(n^5)$ time algorithm.

Though this method of computing the barcode at every slice is not new, using it to efficiently compute the matching distance is. Since our setting is rather different than that of \cite{lw-interactive}, which allows us some simplifications, we give the details of the computation in \cref{Sec:deciding}.

\paragraph{Computing the matching distance~-- previous version.}
We outline the algorithm of \cite{klo-exact}. 
They refine the arrangement $\Ll$ to an arrangement $\Ll'$
by adding $O(n^4)$ additional lines such that within each region, the 
bottleneck distance is determined by the distance of a fixed pair
of simplices. Therefore, the bottleneck function takes a simple
form in each region of $\Ll'$, and they show that in each region,
the bottleneck function is maximized at a boundary vertex (or potentially
as a limit for unbounded cells~-- we ignore this case for simplicity).
Hence, the algorithm simply evaluates the bottleneck function at all
vertices of the arrangement and takes the maximum. This yields
a running time of $O(n^{8+\omega})$ because the arrangement $\Ll'$
has $O(n^8)$ vertices. Using the more efficient barcode
computation from above, the complexity drops to $O(n^{9.5}\log n)$,
because we obtain the barcodes in $O(n^9)$ time and the cost
of computing the bottleneck distance at a fixed dual point is now dominated
by the (geometric) Hopcroft-Karp algorithm with complexity $O(n^{1.5}\log n)$.
Still, the large size of the arrangement $\Ll'$ prevents us from
further improvements with this approach.

\paragraph{Decision version of the matching distance.}
We first study the decision version of the problem: 
Given $\lambda\geq 0$, 
decide whether the matching distance is strictly greater than $\lambda$. 
We show how to answer this question in $\tilde{O}(n^5)$:
for that, we refine the arrangement $\Ll$
to an arrangement $\T_\lambda$ by adding 
$O(n^2)$ further lines. The crucial property is that within each region
of $\T_\lambda$, the bottleneck function $B$ is either $\leq\lambda$
or $>\lambda$ throughout. The construction is based on the observation that
if $B$ changes from $\leq\lambda$ to $>\lambda$ along some path in the dual
plane, this must happen at a dual point where the distance between two simplices along the primal slice changes from $\leq\lambda$
to $>\lambda$. However, we show that such a change can only happen
at dual points that lie on constantly many lines in dual space per pair of simplices~-- note
that some of these dual lines are vertical and therefore do not have
a corresponding point in primal space.

With the arrangement $\T_\lambda$, the decision version is directly solvable
in $O(n^{5.5}\log n)$ time: Pick one point
in the interior of each region, compute the bottleneck function,
and return if the function is $(>\lambda)$ for at least one point.

We further improve the complexity by the same idea of dynamic updates
as for the computation of barcode \change{pairings}. Recall that the question whether
the bottleneck distance between two barcodes is at most $\lambda$
reduces to the existence of a perfect matching in a graph $G^\lambda$.
Every slice/dual point gives rise to such a graph, and by our construction, 
$G^\lambda$ is fixed for every region of $\T_\lambda$.
Moreover, $G^\lambda$ changes in a very controlled way when crossing an
edge of $\T_\lambda$. The most important example is that $G^\lambda$
might gain or lose exactly one edge if we cross a line that changes the 
distance of two simplices from $\leq\lambda$ to $>\lambda$.

This suggests the following strategy: We fix a walk through
$\T_\lambda$ and maintain a perfect matching for $G^\lambda$. 
In many cases, we can either keep the matching or update it in constant time
when walking to the next region. The exception is that an edge of the matching is removed from $G^\lambda$, turning two vertices of $G^\lambda$ unmatched. 
Still, the remaining edges still form an ``almost perfect'' matching.
We simply look for an augmenting path in the new $G^\lambda$ to make the matching
perfect again (if no augmenting path exists, the matching distance is larger
than $\lambda$). In our geometric setting, computing one augmenting path
can be done in $O(n\log n)$ time. So, by starting with an almost-perfect
matching, we avoid the $O(\sqrt{n})$ rounds of the Hopcroft-Karp algorithm
which leads to the improvement in the complexity.

\paragraph{Candidate values.}
To turn the decision version into an algorithm to compute
the matching distance, we identify a candidate set $\lambda_1,\ldots,\lambda_m$
with $m=O(n^6)$ with the guarantee that the exact matching distance $\delta$
is one of the candidates. 

The cleanest way to describe the candidate set is by lifting all 
arrangements $\T_\lambda$ in $\R^3$ parameterized by $(a,b,\lambda)$
such that the intersection with $\lambda=\lambda_0$ yields $\T_{\lambda_0}$. It turns out that this lifting is a plane arrangement
with $O(n^2)$ planes.
We extend this arrangement by adding some other planes and call the resulting arrangement $\T$ (which still has $O(n^2)$ planes).

For every open $3$-dimensional region in $\T$, we have the property
that either $B(a,b)<\lambda$ for all $(a,b,\lambda)$ in the region, or $B(a,b)>\lambda$ for all $(a,b,\lambda)$ in the region.
Call a region \emph{green} in the former and \emph{red} in the latter case.
The matching distance is $\leq\lambda_0$ if and only if the 
the plane $\lambda=\lambda_0$ intersects no red region,
and the exact matching distance $\delta$ is determined by the lowest plane with
that property. It follows that $\delta$ is the $\lambda$-value of a vertex
of $\T$. Since such a vertex is the intersection of $3$ planes,
there are at most $O(n^6)$ such vertices.

Sorting the candidate set and using binary search, we can
then find the exact matching distance using $\log m=O(\log n)$ queries
to the decision algorithm. This results in a $O(n^6\log n)$ algorithm
for the matching distance.

\paragraph{Randomized construction.}
Our final ingredient is to avoid computing all candidates using randomization.
The idea is to iterate through the $\bigO(n^2)$ planes of $\T$ in random order and mantain an interval containing the matching distance. After each iteration, the interval contains the $\lambda$-value of just one vertex that lies on the planes we have considered so far.
After having run through all the planes of $\T$, we are left with just one possible value for the matching distance.

We prove that we can check if the interval needs to be updated in $\tilde\bigO(n^2)$, and if it does, we can update it in $\tilde\bigO(n^4+T(n))$, where $T(n)$ is an upper bound on the complexity of solving the decision problem $d_\M\leq \lambda$. Because of the random order, we only expect $O(\log n)$ updates,
leading to an expected runtime of $\tilde{O}(n^4+T(n))$ in total. As mentioned previously, we will show that $T(n)$ can be taken to be $\tilde\bigO(n^5)$, so we get a total expected runtime of $\tilde\bigO(n^5)$.

\paragraph{Space efficiency.}
The ideas described above yield the desired time complexity, but a space complexity of $O(n^4)$ when realized directly. 
For the decision algorithm, the reason is the construction of the walk through the arrangement which is of length $O(n^4)$.
We can avoid this global walk by replacing it with $O(n^2)$ many local walks, each along one of the lines of the arrangement,
which together cover all regions of the arrangement. These local walks only have a quadratic length and lead
to quadratic space for the decision version.

For the randomized incremental construction, every interval update internally requires a binary search on $O(n^4)$ values
which are determined by an arrangement of $O(n^2)$ lines on the considered plane.
Instead of storing the list of these values and determining their median, we can find an approximate median
with standard methods~\cite[Sec.~9.3]{cormen}, only storing $O(n^2)$ values at every time.

If only the space complexity is of interest, we can compute the matching distance with a conceptually much simpler approach:
we can iterate through the $O(n^6)$ candidate values using only logarithmic space besides storing the input 
because each candidate value is determined
by the intersection of three planes, and each plane is determined by two of the input values.
For each candidate, we apply the decision algorithm and remember the largest candidate
for which the decision algorithm returned true.

Also the decision algorithm can be implemented in linear space: the first major ingredient is that,
for the corresponding line arrangement $\T_\lambda$, we can compute a set of $O(n^4)$ points in the plane, 
using only logarithmic additional space, such that every region of the arrangement contains at least
one of these points. For each point, we compute the bottleneck function and we return true if and only if
the bottleneck function is at most $\lambda$ at every point considered.
The bottleneck function requires the computation of a barcode. The second major ingredient
is that we can compute barcode in linear space, using an unpublished result by Ulrich Bauer.
With that, the decision problem runs in linear space complexity. However, that barcode algorithm
runs in exponential time in the worst case.

\section{Persistence modules and the bottleneck distance}

\paragraph{Persistence modules}

We consider $\R$ and $\R^2$ as posets, $\R$ with the usual poset structure $\leq$, and $\R^2$ with the poset structure given by $(a,b)\leq (a',b')$ if $a\leq a'$ and $b\leq b'$. In what follows, we fix a finite field $F$ and let $P\in \{\R,\R^2\}$.
\begin{definition}
A persistence module $M$ is a collection of vector spaces $M_p$ over $F$ for each $p\in P$ and linear transformations $M_{p\to q}\colon M_p\to M_q$ for all $p\leq q$. We require that $M_{p\to p}$ is the identity and that
\[M_{q\to r}\circ M_{p\to q}= M_{p\to r}\]
whenever these are defined.

If $P=\R$, we call $M$ a $1$-parameter (persistence) module, and if $P=\R^2$, we call $M$ a $2$-parameter (persistence) module.
\end{definition}

We say that two modules $M$ and $N$ are \emph{isomorphic} and write $M\simeq N$ if there is a collection $\{f_p\}_{p\in P}$ of isomorphisms $M_p\to N_p$ such that for all $p\leq q$, $f_q\circ M_{p\to q} = N_{p\to q}\circ f_p$. One can check that the isomorphism relation is an equivalence relation.

The direct sum $M\oplus N$ of modules $M$ and $N$ is defined straightforwardly: $(M\oplus N)_p = M_p\oplus N_p$, and $(M\oplus N)_{p\to q}=M_{p\to q}\oplus N_{p\to q}$.

\begin{definition}
Let $I\subset \R$ be an interval. The interval module $\I^I$ is the $1$-parameter module defined by $\I^I_p=F$ if $p\in I$, and $\I^I_p=0$ if $p\notin I$. In addition, $\I^I_{p\to q}$ is the identity whenever $p,q\in I$.
\end{definition}
We will see below that the $1$-parameter modules we will consider are isomorphic to direct sums of interval modules. The associated intervals give rise to the barcodes mentioned in the introduction.

\paragraph{Presentations}

We will only consider persistence modules that allow a finite description. To be precise, we restrict ourselves to \emph{finitely presented} modules.

Let $P\in \{\R,\R^2\}$, and let $Q=([Q],\G,\Rel,\gr)$ be the collection of the following data:
\begin{itemize}
\item a finite $(m\times m')$-matrix $[Q]$ with entries in $F$,
\item tuples $\G =(g_1,\dots,g_m)$ and 
$\Rel =(r_1,\dots,r_{m'})$ indexing the rows and columns of $Q$, respectively,
\item a \emph{grade} $\gr(g)\in P$ for each $g\in \G$ and $\gr(r)\in P$ for each $r\in \Rel$,
\end{itemize}
such that for every $i$ and $j$ with the entry in the $i^\textrm{th}$ row and $j^\textrm{th}$ column of $[Q]$ nonzero, we have $\gr(g_i)\leq \gr(r_j)$. We call the elements of $\G$ \emph{generators} and those of $\Rel$ \emph{relations}. We often treat $\G$ and $\Rel$ as sets instead of tuples when the order of their elements is understood or irrelevant.

Let $e_1,\dots,e_m$ be the standard basis of $F^m$, $c_j$ the jth column of $[Q]$, and for $p\in P$, define
\begin{align*}
\G_p &=\{e_i\mid \gr(g_i)\leq p \},\\
\Rel_p &=\{c_j\mid \gr(r_j)\leq p \}.
\end{align*}
Then $Q$ induces a module $M^Q$ with
\[M_p^Q = \spn(\G_p)/\spn(\Rel_p)\]
and $M_{p\to q}^Q$ induced by the inclusion $\G_p\to \G_q$. We say that $Q$ is a \emph{presentation} of a module $M$ if $M$ is isomorphic to $M^Q$, and in this case, we say that $M$ is finitely presented. See \cref{Fig:presentation_of_module}.

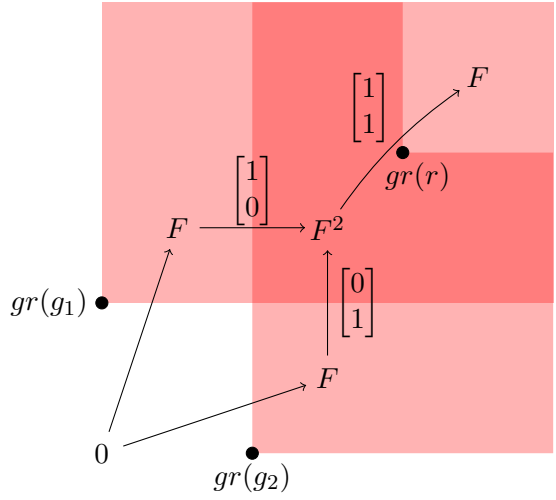
\begin{figure}
\centering
\begin{tikzpicture}
\fill[red, opacity=0.3]
(0,2) to (0,6) to (6,6) to (6,2) to (0,2);
\fill[red, opacity=0.3]
(2,0) to (2,6) to (4,6) to (4,4) to (6,4) to (6,0) to (2,0);
\node at (0,0){$0$};
\node at (1,3){$F$};
\node at (3,1){$F$};
\node at (3,3){$F^2$};
\node at (5,5){$F$};
\draw[->, shorten >=.3cm,shorten <=.3cm] (0,0) to (1,3);
\draw[->, shorten >=.3cm,shorten <=.3cm] (0,0) to (3,1);
\draw[->, shorten >=.3cm,shorten <=.3cm] (1,3) to (3,3);
\draw[->, shorten >=.3cm,shorten <=.3cm] (3,1) to (3,3);
\draw[->, shorten >=.3cm,shorten <=.3cm,out=55,in=215] (3,3) to (5,5);
\node at (2,3.5){$\begin{bmatrix}
1\\
0
\end{bmatrix}$};
\node at (3.4,2){$\begin{bmatrix}
0\\
1
\end{bmatrix}$};
\node at (3.55,4.6){$\begin{bmatrix}
1\\
1
\end{bmatrix}$};
\draw[color=black,fill=black] (0,2) circle (.08);
\node at (-.7,2){$gr(g_1)$};
\draw[color=black,fill=black] (2,0) circle (.08);
\node[below] at (2,0){$gr(g_2)$};
\draw[color=black,fill=black] (4,4) circle (.08);
\node[below] at (4.2,4){$gr(r)$};
\end{tikzpicture}
\caption{A $2$-parameter module presented by $([1;-1],\{g_1,g_2\},\{r\},\gr)$ for certain choices of grades. The shade of red shows the dimension, and some of the vector spaces $M_p$ and linear transformations $M_{p\to q}$ are shown.}
\label{Fig:presentation_of_module}
\end{figure}

We say that a presentation $\partial=([\partial],\G_\partial,\Rel_\partial,\gr)$ is a \emph{permutation of $Q$} if there are permutations $\sigma$ and $\tau$ such that $[\partial]_{i,j}=[Q]_{\sigma(i),\tau(j)}$ for all $i$ and $j$, and $\G_\partial= (g_{\sigma(1)},\dots,g_{\sigma(m)})$ and $\Rel_\partial= (r_{\tau(1)},\dots,r_{\tau(m')})$. That is, we permute the rows and columns of $[Q]$ and permute the indexing elements correspondingly. If $Q$ is a presentation of a module $M$, then so is $\partial$.

A presentation $([Q],\G,\Rel,\gr)$ of a $1$-parameter module is \emph{ordered} if the rows of $[Q]$ are nondecreasing with respect to $\gr(g)$ for $g\in \G$ and the same holds for the columns and $r\in \Rel$. Observe that for any presentation of a $1$-parameter module, there always exists a permutation of it that is ordered.

\paragraph{Input size}

We define the size of a presentation $Q$ as above as $m+m'$; that is, the number of generators and relations.
Now, since the presentation involves an $(m\times m')$-matrix, the actual space needed to store the presentation might be quadratic in $m+m'$.
But often the matrices in question will be sparse, with only $O(m+m')$ nonzero entries, which allows a description of the presentation in $O(m+m')$ space by simply listing the nonzero entries.
For instance, if we compute the $i^\text{th}$ homology of a filtration of simplicial complexes, the matrix will usually be a boundary matrix with at most $i+2$ nonzero entries in each column, as we only need to consider simplices of dimension $\leq i+1$.
We do not know any way to exploit this sparsity to get a more efficient algorithm, so even though our definition sometimes underestimates the actual input size significantly, we would not obtain stronger results if we changed the definition to, say, $m+m'$ plus the number of nonzero entries in $[Q]$.

\paragraph{From presentations to barcodes}

It is a classical result in persistence theory that $1$-parameter modules admit an invariant up to isomorphism called a barcode:
\begin{theorem}[{\cite[Thm.~1.1]{crawley2015decomposition}}]
\label{Thm:barcode_decomposition}
Let $M$ be a finitely presented $1$-parameter module. Then there is a unique multiset $B$ of intervals such that
\[M \simeq \bigoplus_{I\in B} \I^I.\]
\end{theorem}
We call $B$ the barcode of $M$. Isomorphic modules have the same barcode. The original statement of the theorem is for a larger set of modules, but we will not need the stronger version. In our case, where the modules are finitely presented, all the intervals are of the form $[a,b)$ with $a\in \R$ and $b\in \R\cup \{\infty\}$. (Equivalently, one can represent the barcode as a multiset of points in $\R\times (\R\cup\{\infty\})$, with a point $(a,b)$ for each interval $[a,b)$ in the barcode. This alternative description is called a \emph{persistence diagram}, and is used instead of barcodes in some of our references.)

For reasons that will become clear, we need to understand how these barcodes are computed. We give a quick summary following \cite[Sec.~2; Computation]{cem-vines}. First, some terminology: A \emph{column operation} on a matrix is the addition of a multiple of the $i^\textrm{th}$ column to the $j^\textrm{th}$ column for some $i<j$, a \emph{pivot} is a bottommost nonzero element of a (nonzero) column, and \emph{reduced form} means that the pivots are all in different rows. In the standard algorithm to compute the barcode \cite{zomorodian2005computing}, one takes as input an ordered presentation $Q=([Q],\G,\Rel,\gr)$ of a $1$-parameter module $M$ and performs column operations on $[Q]$ until the matrix is in reduced form.

One can summarize this column reduction by $R=[Q] V$, or equivalently $RV^{-1}=[Q]$, where $R$ is the reduced matrix and $V$ is the upper-triangular matrix expressing the column operations we performed on $[Q]$ to get $R$.
We call $(R,V^{-1})$ an \emph{RU-decomposition of $Q$}, also in the case when $Q$ is unordered and/or a presentation of a $2$-parameter module. We consider the rows and columns of $R$ to be indexed by $\G$ and $\Rel$ and view this information as part of the RU-decomposition.

We define the \emph{barcode pairing} associated to $Q$ as follows: It contains each pair $(g,r)$ of a generator $g$ and a relation $r$ such that $\gr(g)\neq\gr(r)$ and the element in the row of $g$ and the column of $r$ in $R$ is a pivot, as well as a pair $(g,\infty)$ for each $g$ whose row in $R$ does not contain a pivot. We view $\infty$ as a dummy relation at infinity, and define $\gr(\infty)=\infty$.
The barcode of $M$ turns out to be the multiset of intervals $[\gr(g),\gr(r))$ for the pairs $(g,r)$ in the barcode pairing.

The RU-decomposition and barcode pairing are not invariants of $M$ (as any module has several presentations), but, as stated by \cref{Thm:barcode_decomposition}, the barcode is. 
If $m$ is the size of $Q$, then an RU-decomposition takes $O(m^2)$ memory, and the barcode pairing $O(m)$ memory.

\paragraph{The bottleneck distance}

We will define the bottleneck distance on $1$-parameter modules in terms of perfect matchings in graphs built on barcode pairings of the modules.
The bottleneck distance is usually defined in terms of barcodes, not barcode pairings, but there are technical reasons why we want to keep track of the pairs of generators and relations instead of throwing that information away and just consider the intervals in the barcode.
The end result is the same, as our definition is equivalent to the usual one.

Let $M$ and $N$ be $1$-parameter modules with barcode pairings $B$ and $B'$, respectively. Given $\lambda\in [0,\infty]$, define an undirected bipartite graph $G=G_{B,B'}^\lambda$ as follows: The vertex set of $G$ is $(B\cup \overline{B'}) \sqcup (B' \cup \overline B)$, where $\overline B$ has an element $\overline{(g,r)}$ for each $(g,r)\in B$, and $\overline{B'}$ is defined similarly. There is an edge between $(g,r)\in B$ and $(g',r')\in B'$ if
\begin{equation}
\label{Eq:grade_diff}
|\gr(g)-\gr(g')|\leq \lambda \quad \text{and} \quad |\gr(r)-\gr(r')|\leq \lambda.
\end{equation}
Let $U\subset B\sqcup B'$ be the set of $(g,r)$ such that
\begin{equation}
\label{Eq:grade_diff2}
|\gr(g)-\gr(r)|\leq 2\lambda.
\end{equation}
For each $(g,r)\in U$, there is an edge between $(g,r)$ and $\overline{(g,r)}$. Lastly, there is an edge between each $\overline{(g,r)}\in \overline B$ and $\overline{(g',r')}\in \overline{B'}$.

Observe that the isomorphism class of $G$ only depends on the barcodes of $M$ and $N$, so the existence of a perfect matching is independent of which barcode pairings we use. This justifies the following definition.
\begin{definition}
\label{Def:Bottleneck_distance}
Let $B$ and $B'$ be barcode pairings of $1$-parameter modules $M$ and $N$. Define the bottleneck distance between $M$ and $N$ as
\[d_\B(M,N) = \min\{\lambda\mid G_{B,B'}^\lambda \text{ allows a perfect matching}\}.\]
\end{definition}
The set we take the minimum of is never empty, as $G_{B,B'}^\infty$ is the complete bipartite graph. That the minimum is well-defined follows from the observation that if an edge is present in $G_{B,B'}^\lambda$ for all $\lambda\in (a,\infty]$, then it is also present in $G_{B,B'}^a$. This, in addition to the fact that as $\lambda$ increases, the vertex set of $G$ stays the same and no edges are removed, implies the following lemma:
\begin{lemma}
\label{Lem:perfect_match_iff_leq_lambda}
$d_\B(M,N)\leq \lambda$ holds if and only if $G_{B,B'}^\lambda$ allows a perfect matching.
\end{lemma}
We remark, leaving the proof to the reader, that the existence of a perfect matching in $G_{B,B'}^\lambda$ is equivalent to the existence of a matching in the subgraph induced by $B\sqcup B'$ covering (a superset of) $(B\sqcup B')\setminus U$. Intuitively, we are allowed to match intervals in the barcodes that are similar, and we do not have to match short (i.e. ``insignificant'') intervals.

\paragraph{The matching distance}

We define a slice as a line $s\subset \R^2$ with positive slope. For $p\in \R^2$, let $p_x$ and $p_y$ be the $x$- and $y$-coordinates of $p$, respectively; that is, $p=(p_x,p_y)$. 
\begin{definition}
Let $M$ be a $2$-parameter module and $s$ a slice with slope at most one. We define $M^s$ as the $1$-parameter module with $M^s_r=M_p$, where $p$ is the unique point on $s$ such that $p_y=r$, and $M^s_{p_y\to q_y}=M_{p\to q}$ for $p\leq q\in s$.

If $s$ has slope more than one, we define $M^s$ as the $1$-parameter module with $M^s_{p_x}=M_{p}$ for all $p\in s$, and morphisms defined by $M^s_{p_x\to q_x}=M_{p\to q}$ for $p\leq q\in s$.
\end{definition}
Thus, if the slope of $s$ is at most one, we can visualize $M^s$ as being the restriction of $M$ to $s$ projected to the $y$-axis. If the slope is more than one, we can visualize $M^s$ as being the restriction of $M$ to $s$ projected to the $x$-axis.

In other definitions of modules analogous to $M^s$ in the literature, one often encounters scaling by weights depending on a parametrization of $s$ or the Euclidean distance in $\R^2$.
The purpose of this rescaling is to ensure that moving forward by $\epsilon$ along $M^s$ corresponds to moving forward by at least $\epsilon$ in both the $x$- and the $y$-coordinate in $\R^2$.
For this paper, we prefer this simpler definition, both because we get this property directly (without going through a different scaling that creates the need for a rescaling), and because the simplicity of projecting to the $x$-/$y$-axis often makes reasoning about $M^s$ easier than with other definitions.

Though $M^s$ is defined in different ways in the literature, there is a standard definition of the matching distance, which is the same as (or equivalent to) ours:
\begin{definition}
We define the matching distance as
\[d_\M(M,N) = \sup_{\text{slice } s}d_\B(M^s,N^s)\]
for finitely presented $2$-parameter modules $M$ and $N$.
\end{definition}
In view of \cref{Lem:Induced_presentation} below, $M^s$ and $N^s$ are finitely presented, so $d_\B(M^s,N^s)$ is indeed well-defined.

We will find it convenient to avoid the case distinction between slope less than and greater than one. Let
\[d_\M^{\leq 1}(M,N) = \sup_{\text{slice } s \text{ with slope }\leq 1}d_\B(M^s,N^s);\]
$d_\M^{\geq 1}(M,N)$ is defined similarly. Clearly,
\[d_\M(M,N)=\max\left\{d_\M^{\leq 1}(M,N),d_\M^{\geq 1}(M,N)\right\},\]
so to compute $d_\M(M,N)$, it is sufficient to compute $d_\M^{\leq 1}(M,N)$ and $d_\M^{\geq 1}(M,N)$.
Since computing $d_\M^{\leq 1}(M,N)$ and $d_\M^{\geq 1}(M,N)$ are completely symmetric problems (to find $d_\M^{\geq 1}(M,N)$, simply switch $x$- and $y$-coordinates of the grades of all generators and relations and apply the algorithm for $d_\M^{\leq 1}$), we can assume without loss of generality that $d_\M^{\leq 1}(M,N)\geq d_\M^{\geq 1}(M,N)$ for the purposes of asymptotic complexity, so that $d_\M^{\leq 1}(M,N)=d_\M(M,N)$. We will use this assumption throughout the rest of the paper.

\section{Deciding if $d_\M\leq \lambda$}
\label{Sec:deciding}

Our goal is to compute $d_\M(M,N)$, but first we consider an easier problem: given $\lambda\geq 0$, decide if $d_\M(M,N)\leq \lambda$.
We explain how to do this in $\bigO(n^5\log n)$.
Here, we assume that the input is a finite presentation of each of $M$ and $N$, and $n$ is the sum of the sizes of these presentations, where we recall that the size of a presentation is defined as its number of generators and relations.
In Section~\ref{sec:deciding_d_M}, we show how to use this to compute $d_\M(M,N)$ exactly.

Fix $\lambda\in [0,\infty)$. Since we assume $d_\M^{\leq 1}(M,N)=d_\M(M,N)$, every slice in this section will be assumed to have slope at most one. To decide if $d_\M(M,N)\leq \lambda$ for modules $M$ and $N$, we need to decide if $d_\B(M^s,N^s)\leq \lambda$ for every slice $s$.
Since there are infinitely many slices, it is not obvious how to do this in finite time, let alone efficiently. To decide if $d_\B(M^s,N^s)\leq \lambda$ for a given slice $s$, the most obvious approach is to find presentations of $M^s$ and $N^s$, then barcode pairings $B$ and $B'$, construct the graph $G_{B,B'}^\lambda$, decide if it has a perfect matching and apply \cref{Lem:perfect_match_iff_leq_lambda}.
Our strategy will be to group the slices into a finite number of sets such that we only need to go through this procedure for one slice in every set. We begin by explaining how a presentation $Q$ of a $2$-parameter module $M$ induces a canonical presentation of $M^s$ for every slice~$s$.

\paragraph{Induced presentations} Let $p\in \R^2$, let $s$ be a slice, and let $q$ be the least point on $s$ that is greater than or equal to $p$. We define $\push_s p$ to be $q_y$.\footnote{NB: In the literature, the push is more often defined as a point on $s$ and not as a point in $\R$, as in our definition.} If $p$ is on or above $s$, then $\push_s p=p_y$; in particular, if $p\in s$ and $M$ is a $2$-parameter module, then $M_{\push_s p}^s=M_p$. If $p$ is below $s$ and $(p_x,r)\in s$, then $\push_s p=r$. See \cref{Fig:pushes}.
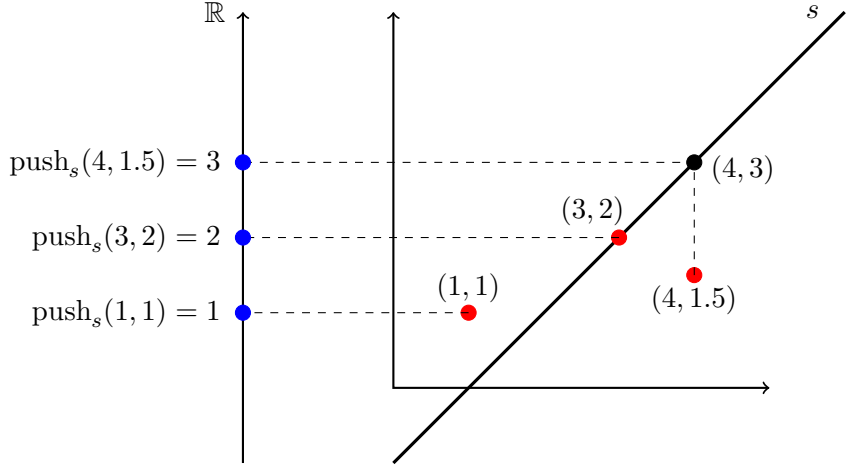
\begin{figure}
\centering
\begin{tikzpicture}
\draw[<->,thick] (5,0) to (0,0) to (0,5);
\draw[->,thick] (-2,-1) to (-2,5);
\node[left] at (-2.1,5){$\R$};
\draw[very thick] (0,-1) to (6,5);
\node[left] at (5.8,5){$s$};
\draw[color=red,fill=red] (3,2) circle (.1);
\node at (2.65,2.35){$(3,2)$};
\draw[dashed] (3,2) to (-2,2);
\draw[color=blue,fill=blue] (-2,2) circle (.1);
\node at (-3.55,2){$\push_s(3,2)=2$};
\draw[color=red,fill=red] (4,1.5) circle (.1);
\node[below] at (4,1.5){$(4,1.5)$};
\draw[color=black,fill=black] (4,3) circle (.1);
\node at (4.65,2.9){$(4,3)$};
\draw[dashed] (4,1.5) to (4,3) to (-2,3);
\draw[color=blue,fill=blue] (-2,3) circle (.1);
\node at (-3.7,3){$\push_s(4,1.5)=3$};
\draw[color=red,fill=red] (1,1) circle (.1);
\node[above] at (1,1){$(1,1)$};
\draw[dashed] (1,1) to (-2,1);
\draw[color=blue,fill=blue] (-2,1) circle (.1);
\node at (-3.55,1){$\push_s(1,1)=1$};
\end{tikzpicture}
\caption{Points in red above, on and below a slice $s$ and the pushes of the points on the real line in blue.}
\label{Fig:pushes}
\end{figure}
Let $Q$ be a presentation of a $2$-parameter module with indexing sets $\G=(g_1,\dots, g_m)$ and $\Rel=(r_1,\dots, r_{m'})$. Let $\push_s(Q)$ be the $1$-parameter presentation with $[\push_s(Q)]=[Q]$, indexing sets \[
\push_s(\G)=(\push_s(g_1),\dots, \push_s(g_m)), \quad \push_s(\Rel)=(\push_s(r_1),\dots, \push_s(r_{m'}))
\]
(here $\push_s(g_i)$ and $\push_s(r_j)$ are to be understood simply as (labels for) the $i^\text{th}$ generator and $j^\text{th}$ relation of $\push_s(Q)$)
and
\[\gr(\push_s(g_i))=\push_s(\gr(g_i)), \quad \gr(\push_s(r_j))=\push_s(\gr(r_j))\]
for all $i$ and $j$. This is indeed a valid presentation, as
\[\gr(p)\leq \gr(q) \implies \gr(\push_s(p))\leq \gr(\push_s(q))\]
for all $p,q\in\R^2$.
\begin{lemma}
\label{Lem:Induced_presentation}
Let $s$ be a slice. If $Q=([Q],\G,\Rel,\gr)$ is a presentation of a $2$-parameter module $M$, then $\push_s(Q)$ is a presentation of $M^s$.
\end{lemma}
\begin{proof}
For $h\in \G\cup \Rel$ and $p\in s$, we have $\gr(h)\leq p$ if and only if $\push_s(\gr(h))\leq p_y$. We get
\begin{align*}
M^{\push_s(Q)}_{p_y} &= \spn\{e_i\mid \push_s(\gr(g_i))\leq p_y \}/\spn\{c_j\mid \push_s(\gr(r_j))\leq p_y \}\\
&= \spn\{e_i\mid \gr(g_i)\leq p \}/\spn\{c_j\mid \gr(r_j)\leq p \}\\
&= \spn(\G_p)/\spn(\Rel_p)\\
&= M^Q_p.
\end{align*}
The morphisms $M^{\push_s(Q)}_{p_y\to q_y}$ and $M^Q_{p\to q}$ are induced in the same canonical way, so they are equal. Since $M^s_{p_y}=M_p$, $M^s_{p_y\to q_y}=M_{p\to q}$ and $M\simeq M^Q$ all hold, we get $M^{\push_s(Q)}\simeq M^s$. Thus, $\push_s(Q)$ is a presentation of $M^s$.
\end{proof}
We say that $\push_s(Q)$ is the presentation of $M^s$ \emph{induced by} $Q$.

\paragraph{RU-decompositions and barcode pairings at slices}

Let $Q=([Q],\G,\Rel,\gr)$ be a presentation of $M$, and let $\partial= ([\partial],\G_\partial,\Rel_\partial,\gr)$ be a permutation of $Q$ such that $\push_s(\partial)$ is an ordered presentation of $M^s$.
Such a permutation always exists, as it is always possible to sort $\push_s(\G)$ and $\push_s(\Rel)$ in nondecreasing order.
If $(R,U)$ is an RU-decomposition of $\partial$, we say that it is an \emph{RU-decomposition of $Q$ at $s$}.
We say that the set $B$ of $(g,r)\in \G_\partial\times \Rel_\partial=\G\times \Rel$ obtained from the pivots of $R$ is a \emph{barcode pairing of $Q$ at $s$}, and let $\push_s(B)$ be the induced barcode pairing of $M^s$.

\paragraph{Computing the barcode at every slice}

The idea of defining RU-decompositions and barcode pairings of a fixed presentation at different slices is that we want to group many slices together and use the same RU-decompositions, barcode pairings and (as we will explain later) bipartite graph and perfect matching for all of them. We will use the following geometric representation exploiting point-line duality to make it easier to work with the set of slices.
\begin{definition}
Let $\omega = (0,1]\times \R$. We say that $(a,b)\in \omega$ represents the slice given by $y=ax+b$.
\end{definition}
We say that a point in $\omega$ and the slice it represents are \emph{dual} to each other. One can see that the set of slices through a point $p$ is represented by the line $y = -p_xx+p_y$. We call this the line \emph{dual to $p$}.
We use the notation $-^*$ for duality -- for instance, $p^*$ is the line given by $y = -p_xx+p_y$, and $(p^*)^*=p$. We say that $p$ lies in the \emph{primal plane}, and $p^*$ in the \emph{dual plane}.
We usually abuse notation and identify a slice with the point in $\omega$ that represents it, but we sometimes write $s^*$ if we want to emphasize that we treat it as a point in the dual plane and not as a line in the primal plane.
The vertical line $x=r$ contains exactly the (points representing) slices with slope $r$.
Thus, $\omega$ contains exactly the slices with slope at most one, which, as mentioned, are the slices we are interested in.
We remind the reader of some standard facts: If $p$ is a point and $s$ a slice, then $p\in s$ if and only if $s^* \in p^*$, and $p$ is above (below) $s$ if and only if $p^*$ is above (below) $s^*$.

For the remainder of the section, we fix two $2$-parameter modules $M$ and $N$ along with presentations $Q=([Q],\G,\Rel,\gr)$ and $Q'=([Q'],\G',\Rel',\gr)$, respectively. Let $\mathfrak G=\G\cup\Rel\cup\G'\cup\Rel'$, and let $n=|\mathfrak G|$.
Our next step is to define a line arrangement in $\omega$ and show that in a certain precise sense, in each region of the arrangement, RU-decompositions and thus barcode decompositions stay constant.
We also show that we can update the RU-decompositions efficiently as we cross a line in the arrangement. Later, as we traverse the arrangement and decide if $d_\B(M^s,N^s)$ at every slice $s$, this will come in handy.

For $p,p'\in \R^2$, let $p\vee p'$ denote the join of $p$ and $p'$, i.e., the least element that is greater than or equal to $p$ and $p'$.
\begin{definition}
Define the following set of lines in $\omega$:
\[\Ll = \{(\gr(h) \vee \gr(h'))^*\mid h,h'\in \mathfrak G\}.\]
\end{definition}
This is the \change{line arrangement $\Ll$} described in \cref{sec:detailed_overview}. There are $\bigO(n^2)$ lines in $\Ll$, so the line arrangement has $\bigO(n^4)$ open regions \cite[Thm.~8.4 (iii)]{dutch}.

For $h\in \mathfrak G$, let $h_x=\gr(h)_x$ and $h_y=\gr(h)_y$. When working with $\Ll$ and the line arrangement $\T_\lambda$ defined below, we run into the technical issue of coinciding lines.
In $\Ll$, this happens for instance if there are $h\neq h'\in \mathfrak G$ with $h_x=h'_x$, as in that case, we get $\gr(h)\vee \gr(h'')=\gr(h')\vee \gr(h'')$ for any $h''\in \mathfrak G$ with $h''_y\geq h_y,h'_y$.
Formally, we handle this by considering the line arrangements as multisets, so for instance we consider $(\gr(h)\vee \gr(h''))^*)$ and $(\gr(h')\vee \gr(h''))^*)$ as distincts elements of $\Ll$, even though they are equal when considered as lines.
This way, each line in $\Ll$ (and $\T_\lambda$) has a \emph{multiplicity} which may be greater than one.
This matters to us because we will traverse these arrangements, and the complexity of our algorithm depends on how many times we have to cross a line.
When lines coincide, we have to count the number of times we cross them with multiplicity, as we need to perform certain operations for each coinciding line.

The following lemma allows us to compute just one RU-decomposition and barcode pairing for each region of $\Ll$, as well as update these efficiently as we move from a region to a neighboring region.
\begin{lemma}
\label{Lem:barcodes_in_arrangement}
Let $S$ be an open region in $\Ll$, let $s\in S$, and let $(R,U)$ be an RU-decomposition of $Q$ at $s$ with associated barcode pairing $B$.
\begin{itemize}
\item[(i)] For any $s'\in S$, $(R,U)$ is an RU-decomposition of $Q$ at $s'$.
\item[(ii)] If $S'$ is another open region in $\Ll$ that is separated from $S$ by a line with multiplicity $k$, then for any $s'\in S$, an RU-decomposition of $Q$ at $s'$ with an associated barcode pairing can be obtained from $(R,U)$ and $B$ in $\bigO(kn)$ time and $O(n^2)$ memory.
\end{itemize}
The analogous statements hold if $Q$ is replaced by $Q'$.
\end{lemma}

\begin{proof}
(i): Any slice $s$ induces a total preorder $\leq_s$ on $\mathfrak G$ by letting $h\leq_s h'$ if $\push_s(\gr(h))\leq \push_s(\gr(h'))$. If (the points represented by) $s$ and $s'$ are both in the same open half-plane defined by $(h\vee h')^*$, then $h \leq_s h'$ if and only if $h \leq_{s'} h'$: If $h_x<h'_x$ and $h_y>h'_y$, then $h<_s h'$ if and only if $s$ is above $h\vee h'$, and $h>_s h'$ if and only if $s$ is below $h\vee h'$. These correspond to $s^*$ lying above $(h\vee h')^*$ and below $(h\vee h')^*$, respectively. We leave the other cases to the reader.

This means that $\leq_s$ only depends on where $s$ is in relation to each line in $\Ll$, i.e., whether it is on, above or below it. Thus, the preorder is constant in $S$. But if $\partial$ is a permutation of $Q$, whether $\push_{s'}(\partial)$ is an ordered presentation of $M^{s'}$ depends only on $\leq_{s'}$. Therefore, $(R,U)$ is an RU-decomposition of $Q$ at $s$ if and only if it is an RU-decomposition of $Q$ at $s'$.

(ii): Suppose $k=1$ and the only line in $\Ll$ separating $S$ and $S'$ is $(\gr(h)\vee \gr(h'))^*$, and let $s\in S$ and $s'\in S'$. We have just shown that $j\leq_s j'$ and $j\leq_{s'} j'$ are equivalent for all $\{j,j'\}\neq \{h,h'\}$. Thus, only the relations between $h$ and $h'$ can change between the two regions. Suppose $\partial$ is a permutation of $Q$ and $\push_s(\partial)$ is an ordered presentation of $M^s$.
In this case, $\push_{s'}(\partial)$ is not necessarily an ordered presentation of $M^{s'}$. But if it is not, we know that we can make the rows and columns of $[\push_{s'}(\partial)]$ nondecreasing by either switching two rows or two columns, namely those indexed by $h$ and $h'$. Suppose $\partial'$ is the permutation of $\partial$ (and thus of $Q$) we get by switching $h$ and $h'$, and their corresponding rows/columns.
By \cite[Sec.~3]{cem-vines}, we can find an RU-decomposition of $\partial'$ with an associated barcode pairing in $O(n)$, since we already have an RU-decomposition of $\partial$ with an associated barcode pairing and we only need to make a single transposition. Because $[\push_{s'}(\partial')]$ is nondecreasing, an RU-decomposition of $\partial'$ is by definition an RU-decomposition of $Q$ at $s'$.
Since $(R,U)$ only takes up $O(n^2)$ memory and the running time is $O(n)$, the memory needed is $O(n^2+n)=O(n^2)$.

If $k>1$, we simply perform the above algorithm for each line in an arbitrary order.
This requires $O(kn)$ time and $O(n^2)$ memory.

The proofs for $Q'$ are exactly the same.
\end{proof}

\paragraph{Deciding if $d_\B\leq \lambda$ at a slice}
We show how deciding if $d_\B(M^s,N^s)\leq \lambda$ for a fixed slice
$s$ boils down to checking if a single graph has a perfect matching.


Let $B\subset \G\times \Rel$ and $B'\subset \G'\times \Rel'$ be barcode pairings at $s$ (so $\push_s(B)$ and $\push_s(B')$ are barcode pairings of $M^s$ and $N^s$).
Recall from \cref{Lem:perfect_match_iff_leq_lambda} that $d_\B(M^s,N^s)\leq \lambda$ if and only if there is a perfect matching in the graph $G_{C,C'}^\lambda$, where $C$ and $C'$ are arbitrary barcode pairings of $M^s$ and $N^s$;
in particular we can choose $C=\push_s(B)$ and $C'=\push_s(B')$. Define $G_{B,B'}^s$ (we write $G_{B,B'}^{s,\lambda}$ if we want to emphasize that it depends on $\lambda$) as being equal to
$G_{\push_s(B),\push_s(B')}^\lambda$, except that each vertex $(\push_s(g),\push_s(r))\in \push_s(B)\cup \push_s(B')$ is replaced by $(g,r)$, and $\overline{(\push_s(g),\push_s(r))}\in \overline{\push_s(B)}\cup \overline{\push_s(B')}$ by $\overline{(g,r)}$.
The graphs $G_{B,B'}^s$ and $G_{\push_s(B),\push_s(B')}^\lambda$ are clearly isomorphic, so one has a perfect matching if and only if the other has.
Thus, $d_\B(M^s,N^s)\leq \lambda$ if and only if there is a perfect matching in $G_{B,B'}^s$.

For $h,h'\in \mathfrak G$, define
\[d^s(h,h')=|\gr(\push_s(h))-\gr(\push_s(h'))|.\]
Observe that in $G_{B,B'}^s$, there is an edge between $(g,r)\in B$ and $(g',r')\in B'$ if and only if $d^s(g,g'),d^s(r,r')\leq \lambda$. Similarly, there is an edge between $(g,r)$ and $\overline{(g,r)}$ if and only if $d^s(g,r)\leq 2\lambda$.

\paragraph{A line arrangement to decide if $d_\B\leq\lambda$}
We now define a line arrangement $\T_\lambda$ of $O(n^2)$ lines. We first
require that $\Ll\subseteq\T_\lambda$, so that in each region
of $\T_\lambda$, any two slices have the same barcode pairings by \cref{Lem:barcodes_in_arrangement}.
 Moreover, we add two sets $\Pp_\lambda$ and $\Ss_\lambda$ of lines to $T_\lambda$ to ensure that within any region,
either $d_\B(M^s,N^s)\leq \lambda$ for every choice of $s$ in the region,
or $d_\B(M^s,N^s)> \lambda$ for every choice of $s$ in the region.

To achieve this, note that $d_\B(M^s,N^s)$ is equal to either $d^s(h,h')$, or $\frac{1}{2}d^s(h,h')$ for some $h,h'\in \mathfrak G$.
\change{(These two possibilities arise from \cref{Eq:grade_diff} and \cref{Eq:grade_diff2}, respectively.)}
For $h$, $h'$ fixed, we want to separate the (dual) slices for which $d^s(h,h')\leq\lambda$ from (dual) slices where $d^s(h,h')>\lambda$. This poses the question of which slices yield $d^s(h,h')=\lambda$.

There are three cases to consider: $h$ and $h'$ can both be above $s$, $h$ can be above and $h'$ below (switch $h$ and $h'$ for the opposite case), or they can both be below $s$.
In the proof of \cref{Lem:equiv_h_h'} below, illustrated in \cref{Fig:point_line_duality}, we show that in the first case, $d^s(h,h')=\lambda$ is equivalent to $s$ passing through $(h_x,h'_y\pm\lambda)$, which implies $s\in \Pp_\lambda$, and in the second case, $d^s(h,h')=\lambda$ is equivalent to $s$ having slope $\frac{\lambda}{|h_x-h'_x|}$, which implies $s\in \Ss_\lambda$.
In the third case, $d^s(h,h')$ is simply $|h_y-h'_y|$ and thus does not depend on $s$, and it turns out that we do not need to add more lines to deal with this case.
Similar statements hold for $d^s(h,h')=2\lambda$.

\begin{definition}
\label{def:lambda_lines}
Define the following sets of lines in $\omega$ for $\lambda>0$:
\begin{align*}
\Pp_\lambda &= \{(h_x,h'_y+i\lambda)^*\mid h,h'\in \mathfrak G, i\in\{-2,-1,1,2\}\},\\
\Ss_\lambda &= \left\{\left\{\frac{i\lambda}{|h_x-h'_x|}\right\}\times \R\mid h,h'\in \mathfrak G, i\in\{1,2\}, |h_x-h'_x|\geq i\lambda \right\},\\
\T_\lambda&=\Pp_\lambda\cup \Ss_\lambda\cup \Ll,
\end{align*}
and the same for $\lambda=0$ except that $\Ss_0=\emptyset$.
\end{definition}
%
For $h,h'\in \mathfrak G$, let $d^-(h,h')\colon \omega\to[0,\infty)$ be the function defined by $s\mapsto d^s(h,h')$.
\begin{lemma}
\label{Lem:equiv_h_h'}
Let $h,h'\in \mathfrak G$. If $s,s'\in \omega$ lie in the same open region of $\T_\lambda$, then $d^s(h,h')$ and $d^{s'}(h,h')$ are contained in the same of the sets $[0,\lambda]$, $(\lambda,2\lambda]$, $(2\lambda,\infty)$.
\end{lemma}
\begin{proof}
Fix $h$ and $h'$. It is straightforward to check that the function $\omega\to \R$ given by $r\mapsto\gr(\push_r(h))$ is continuous. Thus, also $d^-(h,h')$ is continuous.

Assume that $s$ and $s'$ lie in the same open region $S$ of $\T_\lambda$. We will show that if $\lambda\in d^-(h,h')(S)$, then $d^-(h,h')(S)=\{\lambda\}$, and if $2\lambda\in d^-(h,h')(S)$, then $d^-(h,h')(S)=\{2\lambda\}$. By connectivity of $S$ and continuity of $d^-(h,h')$, $d^-(h,h')(S)$ is connected, so this implies that $d^-(h,h')(S)$ is contained in one of the sets $[0,\lambda)$, $\{\lambda\}$, $(\lambda,2\lambda)$, $\{2\lambda\}$, $(2\lambda,\infty)$, which proves the lemma.

Suppose $d^r(h,h')=\lambda$ for some $r\in S$.
We will only consider this case; the argument for $d^r(h,h')=2\lambda$ is practically the same. First assume that $\gr(h)$ and $\gr(h')$ are above or on $r^*$.
(As the arguments that follow are heavy on geometry and point-line duality, we will be careful about differentiating between the point $r$ in $\omega$ and the slice $r^*$ in the primal plane.) Then $d^r(h,h')=|h_y-h'_y|$, so $|h_y-h'_y|=\lambda$.
In this case, $\gr(h)=(h_x,h'_y\pm\lambda)$ and $\gr(h')=(h'_x,h_y\pm\lambda)$, so $\gr(h)^*,\gr(h')^*\in \Pp_\lambda$. Since $\gr(h)$ and $\gr(h')$ are above or on $r^*$ in the primal plane, $r$ is below or on $\gr(h)^*$ and $\gr(h')^*$ in $\omega$.
Since $\gr(h)^*,\gr(h')^*\in \Pp_\lambda$, this implies that all elements of $S$ are below $\gr(h)^*$ and $\gr(h')^*$ in $\omega$, and thus the slices they represent lie below $\gr(h)$ and $\gr(h')$ in the primal plane. Thus, $d^{r'}(h,h')=|h_y-h'_y|=\lambda$ holds for all $r'\in S$.

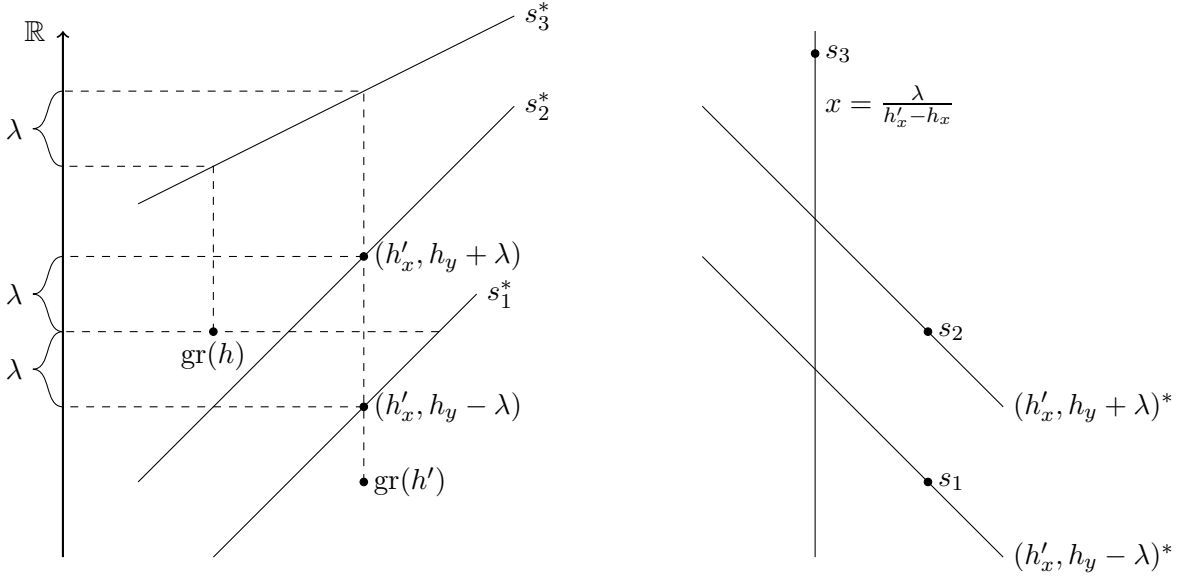
\begin{figure}
\centering
\begin{tikzpicture}
\begin{scope}[]
\draw[->,thick] (-1,0) to (-1,7);
\node[left] at (-1.1,7){$\R$};
\draw[fill=black] (1,3) circle (.05);
\node[below] at (1,3){$\gr(h)$};
\draw[dashed] (4,3) to (-1,3);
\draw[dashed] (1,3) to (1,5.2) to (-1,5.2);
\draw[fill=black] (3,1) circle (.05);
\node[right] at (3,1){$\gr(h')$};
\draw[dashed] (3,1) to (3,6.2) to (-1,6.2);
\draw[dashed] (3,2) to (-1,2);
\draw[dashed] (3,4) to (-1,4);
\draw[fill=black] (3,2) circle (.05);
\node[right] at (3,2){$(h'_x,h_y-\lambda)$};
\draw[fill=black] (3,4) circle (.05);
\node[right] at (3,4){$(h'_x,h_y+\lambda)$};
\draw (1,0) to (4.5,3.5);
\node[right] at (4.5,3.5){$s_1^*\in \Pp_\lambda$};
\draw (0,1) to (5,6);
\node[right] at (5,6){$s_2^*\in \Pp_\lambda$};
\draw (0,4.7) to (5,7.2);
\node[right] at (5,7.2){$s_3^*\in \Ss_\lambda$};
\draw (-1,3) to[out=180,in=0] (-1.4,2.5) to[out=0,in=180] (-1,2);
\node[left] at (-1.4,2.5){$\lambda$};
\draw (-1,4) to[out=180,in=0] (-1.4,3.5) to[out=0,in=180] (-1,3);
\node[left] at (-1.4,3.5){$\lambda$};
\draw (-1,6.2) to[out=180,in=0] (-1.4,5.7) to[out=0,in=180] (-1,5.2);
\node[left] at (-1.4,5.7){$\lambda$};
\end{scope}
\begin{scope}[xshift=7.5cm,yshift=2cm]
\draw[fill=black] (1.5,4.7) circle (.05);
\node[right] at (1.5,4.7){$s_3$};
\draw[fill=black] (3,1) circle (.05);
\node[right] at (3,1){$s_2$};
\draw[fill=black] (3,-1) circle (.05);
\node[right] at (3,-1){$s_1$};
\draw (0,2) to (4,-2);
\node[right] at (4,-2){$(h'_x,h_y-\lambda)^*$};
\draw (0,4) to (4,0);
\node[right] at (4,0){$(h'_x,h_y+\lambda)^*$};
\draw (1.5,5) to (1.5,-2);
\node[right] at (1.5,4){$x=\frac{\lambda}{h'_x-h_x}$};
\end{scope}
\end{tikzpicture}
\caption{On the left: the grades of $h$ and $h'$ in the primal plane, and three slices dual to $s_1,s_2,s_3\in \omega$ with $d^{s_i}(h,h')=\lambda$. $s_1^*$ and $s_2^*$ pass through $(h'_x,h_y\pm\lambda)$, and $s_3^*$ has slope $\frac{\lambda}{h'_x-h_x}$. On the right: $s_1,s_2,s_3$ all lie on lines in $\T_\lambda$. The lines $(h'_x,h_y-\lambda)^*$ and $(h'_x,h_y+\lambda)^*$ are given by the equations $y=-h'_xx+(h_y-\lambda)$ and $y=-h'_xx+(h_y+\lambda)$, respectively.}
\label{Fig:point_line_duality}
\end{figure}

Next, assume that $\gr(h)$ and $\gr(h')$ are below or on $r^*$.
Then there are two points
\[(h_x,a),(h'_x,a\pm \lambda)\in r^*,\]
since $d^r(h,h')=\lambda$.
In this case, $r^*$ has slope $\frac{\lambda}{|h_x-h'_x|}$ and thus $r$ lies on the line $x=\frac{\lambda}{|h_x-h'_x|}$ in $\omega$, which is one of the lines in $\Ss_\lambda$.
(Consider $r=s_3$ in \cref{Fig:point_line_duality}.)
This is a contradiction, as we assumed that $r$ is in an open region $S$, which does not intersect any line in $\T_\lambda$.

Lastly, assume that $\gr(h)$ and $\gr(h')$ are on opposite sides of $r^*$; say $\gr(h)$ is above and $\gr(h')$ below. Then $\gr(\push_r(h)) = h_y$. This means that $\gr(\push_r(h')) = h_y\pm \lambda$, so there is a point $(h'_x,h_y\pm \lambda)\in r^*$. (Consider $r=s_1$ or $r=s_2$ in \cref{Fig:point_line_duality}.) Thus, $r\in (h'_x,h_y\pm \lambda)^*\in \Pp_\lambda$, which is again a contradiction.
\end{proof}

\begin{lemma}
\label{Lem:const_graph_in_region}
Let $S$ be an open region in $\T_\lambda$ and $s\in S$. Let $B\subset \G\times \Rel$ and $B'\subset \G'\times \Rel'$ be barcode pairings at $s\in S$. Then for any $s'\in S$, $G_{B,B'}^{s'}$ is well-defined, and $G_{B,B'}^s=G_{B,B'}^{s'}$.
\end{lemma}

\begin{proof}
Since $\Ll\subset \T_\lambda$, $S$ is contained in an open region of $\Ll$. By \cref{Lem:barcodes_in_arrangement}\,(i), $B$ and $B'$ are then barcode pairings at $s'$, so $G_{B,B'}^{s'}$ is well-defined.

Whether there is an edge between two vertices of $G_{B,B'}^r$ depends only on whether inequalities of the form $d^r(h,h')\leq \lambda$ and $d^r(h,h')\leq 2\lambda$ hold, and by \cref{Lem:equiv_h_h'}, these inequalities hold for $r=s$ if and only if they hold for $r=s'$.
\end{proof}

With this in mind, we define $G_{B,B'}^S$ as $G_{B,B'}^s$ for any $s\in S$ and say that $B$ and $B'$ are barcode pairings \emph{in $S$}. Similarly, we say that an RU-decomposition at $s\in S$ is an RU-decomposition \emph{in $S$}, as justified by \cref{Lem:barcodes_in_arrangement}\,(i).
\begin{lemma}
\label{Lem:Neighboring_matching}
Let $R$ and $S$ be regions in $\T_\lambda$ that are separated by $k$ lines counted with multiplicity, and suppose we are given RU-decompositions with associated barcode pairings $B\subset \G\times \Rel$ and $B'\subset \G'\times \Rel'$ in $R$ and a perfect matching in $G_{B,B'}^R$. Then we can find RU-decompositions with associated barcode pairings $B_S\subset \G\times \Rel$ and $B_S'\subset \G'\times \Rel'$ in $S$ and either find a perfect matching in $G_{B_S,B_S'}^S$ or determine that a perfect matching does not exist in $\bigO(kn\log n)$ time and $O(n^2)$ memory.
\end{lemma}
\begin{proof}
We first assume $k=1$ and that the line $\ell$ separating $R$ and $S$ is in $\Ll$. Then we can update the RU-decompositions and thus $B$ and $B'$ in $\bigO(n)$ using \cref{Lem:barcodes_in_arrangement}\,(ii). By \cite[Fig. 4, left]{cem-vines}\footnote{The setup in \cite{cem-vines} is a little different than ours. In their framework, an element can act both as a generator and a relation. Since this is impossible for us, we can ignore the right part of Fig. 4.}, either $B$ and $B'$ do not change, or one of the following happens:
\begin{itemize}
\item a pair $(g,r)$ appears or disappears in $B$ or $B'$ (in this case, the RU-decomposition does not necessarily change, but $(g,r)$ gives rise to an empty interval in $R$ and a nonempty interval in $S$, or vice versa),
\item two pairs $(g,r)$ and $(g',r')$ are replaced by $(g',r)$ and $(g,r')$ in $B$ or $B'$.
\end{itemize}
In both cases, the graph $G_{B,B'}^R$ only has to be modified ``locally'' to get a graph $G_{B_S,B_S'}^S$: At most two vertices are added, removed or replaced, and only the edges adjacent to the modified vertices are changed. (Actually, in the second bullet point, the graphs are isomorphic. We will not need this fact, so we do not prove it.) The matching in $G_{B,B'}^R$ therefore gives us a matching in $G_{B_S,B_S'}^S$ covering all but at most four vertices.

Now assume that $\ell$ is in $\T_\lambda\setminus\Ll=\Pp_\lambda\cup \Ss_\lambda$. Then by \cref{Lem:barcodes_in_arrangement}\,(i), the RU-decompositions and barcode pairings are valid also in $S$, as $R$ and $S$ are contained in the same region of $\Ll$. The only relevant change crossing $\ell$ can make is that there might be two elements $h,h'\in \mathfrak G$ such that $d^r(h,h')\leq\lambda$ and $d^s(h,h')>\lambda$ or vice versa for $r\in R$ and $s\in S$, or the same statement holds for $2\lambda$ instead of $\lambda$. The only change this can make to the graph is that an edge between the elements in $B\cup B'$ involving $h$ and $h'$ has to be added or removed, or, in the case $(h,h')\in B\cup B'$, an edge between $(h,h')$ and $\overline{(h,h')}$ has to be added or removed. In either case, the matching in $G_{B,B'}^R$ induces a matching in $G_{B,B'}^S$ (which is well-defined, as $B$ and $B'$ are valid in $S$) covering all but at most two vertices.

Thus, we have reduced the problem to finding a perfect matching, or decide that a perfect matching does not exist, in a graph $G_{B_S,B_S'}^S$ where we already have a matching covering all but at most four vertices.
By \cite[Thm.~3.1]{kmn-geometry}, which builds on \cite{eik-geometry}, the bottleneck distance can be computed in $\bigO(n^{1.5}\log n)$. But a factor of $O(n^{0.5})$ comes from the number of times one has to augment the matching, and after each augmentation one has a matching with at least one more edge.
Since we only need to add a constant number of edges, we can replace this factor by a constant, so the complexity of constructing a perfect matching or determining that it does not exist is only $\bigO(n\log n)$.

For the case $k>1$, we can do everything as above for each line, except that we might not obtain a perfect matching at every intermediate step.
Therefore, we do not update the matching until we have crossed all the lines, at which point we have a matching covering all but at most $4k$ vertices in $G_{B_S,B_S'}^S$.
By the same argument as above, we can find a perfect matching in $G_{B_S,B_S'}^S$ or determine that such a matching does not exist in at most $\bigO(kn\log n)$.

Regarding memory, we have observed that RU-decompositions and barcode matchings can be stored in $O(n^2)$.
If $k=1$, updating these can only require $O(n^2+n)=O(n^2)$ memory, as the running time is $O(n)$ for such an update.
Doing $k>1$ such updates in a row does not change this bound for the memory.
The memory required for updating the matchings cannot be larger than $\bigO(n^{1.5}\log n)$, as we are only executing parts of an algorithm that that runs in $\bigO(n^{1.5}\log n)$ on an input size of $n$.
\end{proof}
\begin{theorem}
\label{Thm:runtime_leq_lambda}
For any $\lambda\in[0,\infty)$, we can decide if $d_\M(M,N)\leq \lambda$ in $\bigO(n^5 \log n)$ and memory $O(n^2)$.
\end{theorem}
\begin{proof}
By \cite[Lemma 3]{klo-exact}, the function $f$ sending a slice $s$ to $d_\B(M^s,N^s)$ is continuous. Since the union $U$ of open regions of $\T_\lambda$ is dense in $\omega$, we have $f(U)\subset [0,\lambda]$ if and only if $f(\omega)\subset [0,\lambda]$. Thus, to find out if $d_\M\leq \lambda$, it is enough to check if $d_\B(M^s,N^s)\leq \lambda$ for all $s$ in open regions. By \cref{Lem:perfect_match_iff_leq_lambda} and \cref{Lem:const_graph_in_region}, for any open region $S$ and $s,s'\in S$, $d_\B(M^s,N^s)\leq \lambda$ is equivalent to $d_\B(M^{s'},N^{s'})\leq \lambda$.
Thus, it is enough to check if $d_\B(M^s,N^s)\leq \lambda$ for one slice $s$ in each open region.


We begin by listing the lines in $\T_\lambda$ in any order; let this list be $\ell_1,\dots,\ell_m$ ($O(n^2)$ time and memory).
Then we pick any region $S$ adjacent to $\ell_1$ (by which we mean that \change{$\ell_1$} intersects the closure of $S$) and compute RU-decompositions of $Q$ and $Q'$ by column operations in $S$, find the associated barcode pairings $B$ and $B'$, construct $G_{B,B'}^S$ and either find a perfect matching or determine that it does not exist with the Hopcroft-Karp algorithm \cite{hopkarp}.
All of this can be done in $\bigO(n^3)$ time and $O(n^2)$ memory.
If there is no perfect matching, then $d_\B(M^s,N^s)>\lambda$ for all $s\in S$, so $d_\M(M,N)>\lambda$, and we are done.
Assuming that we found a perfect matching, we have $d_\B(M^s,N^s)\leq\lambda$ for all $s\in S$.

Next, we start the first out of $m$ iterations of the following algorithm:
\change{At step $i\geq 1$,} we assume that we already have RU-decompositions, barcode pairings and a perfect matching at a region $R$ adjacent to $\ell_i$.
We sort the lines $\ell_1,\dots, \ell_{i-1}, \ell_{i+1},\dots,\ell_m$ by when they cross $\ell_i$, say, in lexicographic order.
Starting at $R$, we then traverse the $O(n^2)$ regions adjacent to $\ell_i$ clockwise around the line segment between the top and bottom vertex on $\ell_i$; see \cref{Fig:Arrangement_walk}.
Counted with multiplicity, we have to cross $O(n^2)$ lines.
By \cref{Lem:Neighboring_matching}, we can update the RU-decompositions, barcode pairings and perfect matching (or determine that it does not exist) in $\bigO(kn\log n)$ time as we cross a line with multiplicity $k$, which gives a running time for going through all the regions adjacent to $\ell_i$ of $O(n^3\log n)$.
While doing this, we will visit a region intersecting $\ell_{i+1}$ (unless $i=m$), and at this region, we store the RU-decompositions, etc. to be ready for iteration $i+1$ of the algorithm.
At any given point, we only have to store a constant number of RU-decompositions, barcode pairings and maximal matchings, which takes $O(n^2)$ memory.

As any region is adjacent to a line, we visit every region of $\T_\lambda$ in this algorithm.
If $d_\B(M^s,N^s)\leq\lambda$ for all $s\in R$ for each region $R$, then $d_\M(M,N)\leq \lambda$; otherwise, we stop once we have found a region in which the inequality fails and conclude that $d_\M(M,N)>\lambda$.
Thus, we have computed what we wanted.
The total running time is $O(n^2\cdot n^3\log n)=O(n^5\log n)$, and as explained above, the memory consumption is never larger than $O(n^2)$.

\begin{figure}
\centering
\begin{tikzpicture}[scale=.8]
\begin{scope}
\draw (0,0) to (5,1.5);
\node[left] at (0,0){$\ell_{i+1}$};
\draw (5,0) to (0,5);
\node[left] at (0,5){$\ell_i$};
\draw (1,-.5) to (4,5);
\draw (0,3) to (5,4);
\draw (0,1) to (2.86,5);
\draw[->, dashed, shorten >=.1cm,shorten <=.1cm] (3,4.2) to (2.3,3);
\draw[->, dashed, shorten >=.2cm,shorten <=.1cm] (2.3,3) to (4,2.2);
\draw[->, dashed, shorten >=.1cm,shorten <=.1cm] (4,2.2) to (4.9,.8);
\draw[->, dashed, shorten >=.1cm,shorten <=.1cm] (4.9,.8) to (3.5,0.5);
\draw[->, dashed, shorten >=.1cm,shorten <=.1cm] (3.5,0.5) to (2.6,1.4);
\draw[->, dashed, shorten >=.1cm,shorten <=.1cm] (2.6,1.4) to (1.4,2.3);
\draw[->, dashed, shorten >=.1cm,shorten <=.1cm] (1.4,2.3) to (.5,2.8);
\draw[->, dashed, shorten >=.1cm,shorten <=.1cm] (.5,2.8) to (0.5,3.8);
\draw[->, dashed, shorten >=.1cm,shorten <=.1cm] (0.5,3.8) to (1.8,4.5);
\draw[black,fill=black] (3,4.2) circle (.15);
\draw[black,fill=black] (3.85,2.05) rectangle (4.15,2.35);
\end{scope}
\end{tikzpicture}
\caption{Step $i$ of our traversal of $\T_\lambda$.
We visit every region adjacent to $\ell_i$ clockwise from the initial region marked with a disk, and when we are in a region adjacent to $\ell_{i+1}$, marked with a square, we store the information needed to perform step $i+1$ of the traversal.}
\label{Fig:Arrangement_walk}
\end{figure}
\end{proof}

\section{Computing $d_\M$}
\label{sec:deciding_d_M}

From now on, we use the shorthand $d_\M$ for $d_\M(M,N)$. 
As in the previous section, $n$ denotes the sum of generators
and relations of $M$ and $N$.
Before we explain the algorithm to compute $d_\M$, we deal with the special cases $d_\M=0$ and $d_\M=\infty$. Let $(T(n),M(n))$ be upper bounds for the time and space complexity of an algorithm to decide if $d_\M\leq \lambda$ for an arbitrary $\lambda$. (In the last section we proved that we can take $(T(n),M(n))$ to be $(\bigO(n^5\log n),\bigO(n^2))$.)
First, we check if $d_\M\leq 0$ in time $T(n)$ and space $M(n)$. If the answer is yes, we know that $d_\M=0$. To find out if $d_\M=\infty$, we check if $d_\B(M^s,N^s)=\infty$ for one slice $s$ in $\bigO(n^3)$.
These are equivalent by the lemma below.
Considering this, we will assume $d_\M\in (0,\infty)$.
\begin{lemma}
For any slice $s$, $d_\B(M^s,N^s)=\infty$ is equivalent to $d_\M(M,N)=\infty$.
\end{lemma}
\begin{proof}
By definition, $d_\B(M^s,N^s)=\infty$ implies $d_\M(M,N)=\infty$.
Suppose that $d_\M(M,N)=\infty$ and that there is a slice $s$ such that $d_\B(M^s,N^s)$ is finite.
Since the function sending a slice $s$ to $d_\B(M^s,N^s)$ is continuous, as noted above, this means that $d_\B(M^s,N^s)$ is finite for all $s$.
The only way $d_\M(M,N)$ can be infinite in this case is if there are slices $s$ such that $d_\B(M^s,N^s)$ is arbitrarily large, in particular larger than $K$, where $K$ is a constant such that $\gr(h)\in \left[-\frac{K}{2},\frac{K}{2}\right]^2$ for all $h\in \mathfrak G$.
In this case, if $B$ and $B'$ are barcode pairings at $s$, then $G_{B,B'}^{s,K}$ does not have a perfect matching, but $G_{B,B'}^{s,K'}$ does for $K'$ large enough.
But this is impossible, as one can check that for any $h,h'\in \mathfrak G$,
\[
|d^s(h,h')-d^{s'}(h,h')|\leq \max\left\{|h_x-h'_x|,|h_y-h'_y|\right\} \leq K,
\]
which means that $G_{B,B'}^{s,K}$ is already a complete bipartite graph and thus isomorphic to $G_{B,B'}^{s,K'}$.
\end{proof}

\paragraph{The plane arrangement $\T$}

We will now move up a dimension from $\omega$ and work with $\Omega \coloneqq \omega\times [0,\infty)$. Elements in $\Omega$ are of the form $(s,\lambda)$, where $s$ should be viewed as a slice and $\lambda$ as a potential value for $d_\M$. We now define a plane arrangement $\T= \Pp\cup \Ss\cup \Hh$ that can be viewed as the union of the line arrangements $\T_\lambda$ for all $\lambda$ with some extra planes added in.
More precisely, the intersection of the planes in $\Pp$ with $\R^2\times \{\lambda\}$ gives a set of lines that contains all the lines in $\Pp_\lambda$, assuming that we identify $\R^2\times \{\lambda\}$ with $\R^2$ in the obvious way.
We see from the definition of $\Pp$ below that this set also contains the lines in $\Ll$, as $(h_1\vee h_2)^*\in\Ll$ is of the form $(h_x,h'_y+i\lambda)^*$ for $i=0$ and some choice of $h,h'\in \{h_1,h_2\}$.
Similarly, one can get all the lines in $\Ss_\lambda$ by intersecting the planes in $\Ss$ with $\R^2\times \{\lambda\}$.
We also add two planes $S_0$ and $S_1$ that contain parts of the boundary of $\Omega$, and a set $\Hh$ of planes of the form $\R^2\times \{\lambda\}$ for some $\lambda$ that need to be treated separately.
We define these planes as subsets of $\R^3$. $\Omega$ is viewed as a subset of $\R^2\times \R\simeq \R^3$. 

\begin{definition}
\label{def:plane_arrangement}
We define the following sets of planes in $\Omega$. Lines of the form $p^*$ should be understood to be the extension of $p^*$ in $\omega$ to $\R^2$.
\[\Pp = \{P_{h,h',i}\mid h,h'\in \mathfrak G, i\in\{-2,-1,0,1,2\}\},\]
where
\[P_{h,h',i} = \bigcup_{\lambda\in\R}(h_x,h'_y+i\lambda)^*\times \{\lambda\};\]
\[\Ss = \{S_{h,h',i}\mid h,h'\in \mathfrak G, |h_x-h'_x|\neq 0, i\in\{1,2\}\}\cup \{S_0,S_1\},\]
where
\[S_{h,h',i} = \left\{((a,b),\lambda)\mid a = \frac{i\lambda}{|h_x-h'_x|}\right\},\]
\[S_i = \{i\}\times\R^2;\]
\[\Hh = \{\R^2\times \{\lambda\}\mid \exists (h,h')\in \mathfrak G \text{ s.t. } |h_y-h'_y| \in \{\lambda,2\lambda\} \}.\]
Finally, define $\T= \Pp\cup \Ss\cup \Hh$.
\end{definition}
Note that $|\T |=\bigO(n^2)$.

It is clear that the sets in $\Ss$ and $\Hh$ are indeed planes, as they are given by linear equations. The same holds for the sets in $\Pp$, as $((a,b),\lambda)\in P_{h,h',i}$ is equivalent to $b=-h_x a+h'_y+i\lambda$.

For a set $\X$ of planes, we write $\bigcup \X = \bigcup_{H\in \X} H$. We call a connected component of the complement of $\UT$ an \emph{open cell}. 
Letting $h=h'$, we see that $\R^2\times\{0\}\in\Hh$. The union $\R^2\times\{0\}\cup S_0 \cup S_1$ contains the boundary of $\Omega$. Therefore, any open cell is either contained in $\Omega$ or contained in $\R^3\setminus\Omega$.

\paragraph{$d_\M$ is attained at a vertex of $\T$}

Define a function $f\colon\Omega\to\R$ by
\[(s,\lambda) \mapsto \lambda-d_\B(M^s,N^s)
.\]
Thus, $d_\B(M^s,N^s)\leq \lambda$ if and only if $f(s,\lambda)\geq 0$.

\begin{proposition}
$f$ is continuous.
\end{proposition}
\begin{proof}
As we have already noted, \cite[Lemma 3]{klo-exact} shows that the function sending a slice $s$ to $d_\B(M^s,N^s)$ is continuous. It follows that $f$ is continuous.
\end{proof}

\begin{lemma}
\label{Lem:zeros_are_in_UT}
Let $(s,\lambda)\in \Omega$. If $f(s,\lambda)=0$, then $(s,\lambda)\in \UT$.
\end{lemma}
\begin{proof}
If $f(s,\lambda)=0$, there is an edge in $G_{B,B'}^{s,\lambda}$ that is not in $G_{B,B'}^{s,\lambda'}$ for any $\lambda'<\lambda$ (where $B$ and $B'$ are any barcode pairings of $M^s$ and $N^s$). By the definition of the graphs, this means that there are $h,h'\in \mathfrak G$ such that $d^s(h,h')\in\{\lambda,2\lambda\}$.
We showed in the proof of \cref{Lem:equiv_h_h'} that this implies that either $s\in \ell$ for some $\ell\in \T_\lambda$, or $|h_y-h'_y|\in\{\lambda,2\lambda\}$. In the first case, $s\in \bigcup(\Pp\cup \Ss)$, and in the second, $s\in \bigcup\Hh$.
\end{proof}
\begin{corollary}
	\label{Cor:f_pos_neg}
Suppose $C\subset \Omega$ is an open cell. Either $f(s,\lambda)> 0$ for all $(s,\lambda)\in C$, or $f(s,\lambda)< 0$ for all $(s,\lambda)\in C$.
\end{corollary}
\begin{proof}
Since $C$ is connected and $f$ is continuous, $f(C)$ is connected. By \cref{Lem:zeros_are_in_UT}, $0\notin f(C)$, so either $f(C) \subset (0,\infty)$ or $f(C) \subset (-\infty,0)$.
\end{proof}
By definition,
\begin{align*}
d_\M &= \sup_{\text{slice } s}d_\B(M^s,N^s)\\
&=\sup\{\lambda \mid \exists s \text{ such that } f(s,\lambda)\leq 0 \}\\
&=\sup\{\lambda \mid \exists s \text{ such that } f(s,\lambda)< 0 \}\\
&=\sup\{\lambda \mid \exists s \text{ such that } (s,\lambda)\in f^{-1}(-\infty,0) \}.
\end{align*}
Since $d_\M>0$ by assumption, there exists an $s$ such that $d_\B(M^s,N^s)>0$, so none of the sets we are taking the supremum over is empty. Since $\Omega \setminus \UT$ (i.e., the union of all open cells in $\Omega$) is dense in $\Omega$ and $f^{-1}(-\infty,0)$ is open, this is equivalent to
\[d_\M = \sup\{\lambda \mid \exists s \text{ such that } (s,\lambda)\in f^{-1}(-\infty,0)\cap (\Omega \setminus \UT) \}.\]
By \cref{Cor:f_pos_neg}, it follows that
\begin{equation}
\label{Eq:dM_max_cells}
d_\M \in \left\{\sup_{(s,\lambda)\in C}\lambda\mid C \text{ open cell in }\Omega\right\} \subset \left\{\sup_{(s,\lambda)\in C}\lambda\mid C \text{ open cell}\right\},
\end{equation}
where we now assume that $s$ is a point in $\R^2$ that is not necessarily in $\omega$. Call a cell \emph{small} if its projection onto $\R$ is bounded above. Since we have assumed $d_\M\neq \infty$, we can restrict ourselves to small cells:
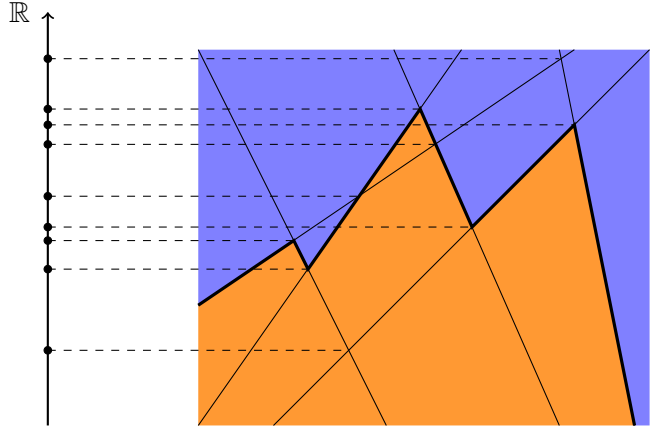
\begin{figure}
\centering
\begin{tikzpicture}
\draw[->,thick] (-2,-1) to (-2,4.5);
\node[left] at (-2.1,4.5){$\R$};
\fill[orange, opacity=0.8]
(0,.6) to (1.27,1.46) to (1.46,1.08) to (2.95,3.21) to (3.64,1.64) to (5,3) to (5.8,-1) to (0,-1);
\fill[blue, opacity=0.5]
(0,.6) to (1.27,1.46) to (1.46,1.08) to (2.95,3.21) to (3.64,1.64) to (5,3) to (5.8,-1) to (6,-1) to (6,4) to (0,4);
\draw[very thick] (0,.6) to (1.27,1.46) to (1.46,1.08) to (2.95,3.21) to (3.64,1.64) to (5,3) to (5.8,-1);
\draw (0,-1) to (3.5,4);
\draw (2.5,-1) to (0,4);
\draw (1,-1) to (6,4);
\draw (4.8,-1) to (2.6,4);
\draw (5.8,-1) to (4.8,4);
\draw (0,.6) to (5,4);
\draw[dashed] (-2,1.46) to (1.27,1.46);
\draw[fill=black] (-2,1.46) circle (.05);
\draw[dashed] (-2,1.08) to (1.46,1.08);
\draw[fill=black] (-2,1.08) circle (.05);
\draw[dashed] (-2,3.21) to (2.95,3.21);
\draw[fill=black] (-2,3.21) circle (.05);
\draw[dashed] (-2,1.64) to (3.64,1.64);
\draw[fill=black] (-2,1.64) circle (.05);
\draw[dashed] (-2,3) to (5,3);
\draw[fill=black] (-2,3) circle (.05);
\draw[dashed] (-2,3.88) to (4.8,3.88);
\draw[fill=black] (-2,3.88) circle (.05);
\draw[dashed] (-2,2.74) to (3.15,2.74);
\draw[fill=black] (-2,2.74) circle (.05);
\draw[dashed] (-2,2.05) to (2.14,2.05);
\draw[fill=black] (-2,2.05) circle (.05);
\draw[dashed] (-2,0) to (2,0);
\draw[fill=black] (-2,0) circle (.05);
\end{tikzpicture}
\caption{For ease of visualization, we illustrate $\T_\lambda$ as a line arrangement in $\R^2$ instead of as a plane arrangement in $\R^3$. The function $f$ is negative in the orange cells and nonnegative in the blue cells. The levels of the vertices are shown on the real line on the left. These are the potential values for $d_\M$, which is the supremum of the projection of the orange region to the real line. In $\T_\lambda$, the vertices are triple intersections of planes instead of (double) intersections of lines.}
\label{Fig:vertices_and_sign_of_f}
\end{figure}
\begin{equation}
\label{Eq:dM_max_cells2}
d_\M \in \left\{\sup_{(s,\lambda)\in C}\lambda\mid C \text{ small open cell}\right\} = \left\{\max_{(s,\lambda)\in \overline C}\lambda\mid C \text{ small open cell}\right\}.
\end{equation}
($\overline C$ denotes the closure of $C$.) Each value $\max_{(s,\lambda)\in \overline C}\lambda$ must be attained at a point that is the unique intersection point between three planes.\footnote{This requires that not all the planes in $\T$ are parallel to a single line. Since $M$ and $N$ are not both the zero module, $\mathfrak G$ contains a generator $g$. Then there are three planes $P_{g,g,1}$, $S_0$ and $\R^2\times \{0\}$ in $\T$ that are not all parallel to one line.} Let the \emph{level} of a point in $\R^3$ be its third coordinate. That is, the level of $(a,b,\lambda)$ is $\lambda$. By a \emph{vertex of $\T$}, we mean a point $p$ such that there exist $P,P',P''\in\T$ with $P\cap P' \cap P''=\{p\}$. See \cref{Fig:vertices_and_sign_of_f}. We sometimes abuse terminology and call $\{p\}$ a vertex when we mean that $p$ is a vertex.

\paragraph{Searching through the vertices to find $d_\M$}

Recall that we assume that $(T(n),M(n))$ are upper bounds on the time and memory needed to decide $d_\M\leq \lambda$ for a given $\lambda$, and that we proved that we can take $T(n)$ to be $\bigO(n^5\log n)$ and $M(n)$ to be $\bigO(n^2)$ in the previous section. To find $d_\M$, we could do the following: Compute the intersections of all triples of planes in $\T$ (for the triples whose intersection is a single point), sort the levels of these points, and find $d_\M$ by binary search. This would give us a runtime of $\bigO(n^6\log n + T(n)\log n)=\bigO(n^6\log n)$, as the most expensive operation is to sort the $\bigO(n^6)$ triple intersections. Also, this approach would require $\bigO(n^6)$ memory to store and sort all intersections.

However, we want to do better and describe an algorithm that runs in expected time
\[\bigO((n^4 + T(n))\log^2 n).\]
and with $\bigO(n^2+M(n))$ memory.
Roughly, our idea is to start with an interval $(a,b]$ that contains $d_\M$ and run through the planes in $\T$ in random order, narrowing the interval as we go. After having dealt with a plane $P\in \T$, we want the updated $a$ and $b$ to be such that the interval $(a,b)$ does not contain the level of any vertex in $P$.
In the end, we will be left with an interval $(a,b]$ such that $b$ is the only point in the interval that is the level of a vertex, so we can conclude that $d_\M=b$. By a randomization argument, we will prove that the expected number of times we have to update $a$ and $b$ is $\bigO(\log n)$.
Thus, since $|\T|=\bigO(n^2)$, it is sufficient to prove that we can check if we have to update $a$ and $b$ for a given plane in $\bigO(n^2\log n)$ (also $\bigO(n^2\log^2 n)$ would be good enough), and that we do not spend more time than $\bigO((n^4+T(n))\log n)$ every time we update $a$ and $b$.
\begin{definition}
\label{Def:I_P}
For a plane $P\in\T$, let $c_1<\dots<c_K$ be the levels of the vertices of $\T$ in $P$. Let $c_{K+1}=\infty$, and if $c_1>0$, let $c_0=0$. Let $I_P=(\alpha_P,\beta_P]$ be the interval of the form $(c_i,c_{i+1}]$ that contains $d_\M$.
\end{definition}
\begin{lemma}
\label{Lem:decide_inclusion}
For any $P\in\T$ and any interval $(\alpha,\beta]$ containing $d_\M$, we can decide if $(\alpha,\beta]\subseteq I_P$ in $\bigO(n^2\log n)$ time and $\bigO(n^2)$ memory.
\end{lemma}
The idea of the proof is illustrated in \cref{Fig:P_alpha_beta}: $\alpha$
and $\beta$ induce two parallel lines on $P$, and the planes in $T$ not parallel to $P$ induce a line arrangement on $P$. Now, $(\alpha,\beta]\subseteq I_P$
if and only if this line arrangement does not have an intersection point
in the strip between $\alpha$ and $\beta$. This in turn is equivalent
to the order of these arrangement lines along the $\alpha$ line
and along the $\beta$-line being the same, which can be checked
in $\bigO(n^2\log n)$ time and $\bigO(n^2)$ memory by sorting. The subsequent proof contains
more details:
\begin{proof}
If $P$ is of the form $\R^2\times\{\lambda\}$, then all the vertices in $P$ are at level $\lambda$. Thus, $I_P$ is either $(0,\lambda]$ or $(\lambda,\infty]$, so $(\alpha,\beta]\subseteq I_P$ holds if and only if $\lambda\notin (\alpha,\beta)$, which we can check in constant time. (The intersection $(\alpha,\beta]\cap I_P$ cannot be empty, as both intervals contain $d_\M$.)

Assume $P$ is not of the form $\R^2\times\{\lambda\}$. Then $P$ intersects the planes $\R^2\times\{\alpha\}$ and $\R^2\times\{\beta\}$ in lines $P_\alpha$ and $P_\beta$, respectively. Let $P_{(\alpha,\beta)}$ be the subset of $P$ that lies strictly between $P_\alpha$ and $P_\beta$. We have $(\alpha,\beta]\subseteq I_P$ if and only if there are no vertices in $P_{(\alpha,\beta)}$.

Call a plane $P'\in \T$ \emph{straight} if $P\cap P'$ is a line contained in a plane of the form $\R^2\times\{\lambda\}$. That is, all the points in $P\cap P'$ are at the same level. Call a plane in $\T$ \emph{slanted} if it is neither straight nor parallel to $P$.

We first find all the straight planes in $\T$ and check if the level of the intersection of each plane with $P$ (i.e., the level of any point in this intersection) is strictly between $\alpha$ and $\beta$. We can do this in $\bigO(n^2)$ time and memory. If we find one such plane $P'$, then all the vertices in $P\cap P'$ are at a level between $\alpha$ and $\beta$, and we conclude that $(\alpha,\beta]\nsubseteq I_P$. Otherwise, we know that no straight plane contains a vertex in $P_{(\alpha,\beta)}$, so we can ignore the straight planes.

Assuming that the straight planes can be safely disregarded, the only vertices we have left to check are those of the form $P\cap P'\cap P''$ where $P'$ and $P''$ are both slanted. If $P'$ is slanted, $P'\cap P_\alpha$ and $P'\cap P_\beta$ both contain exactly one point. This is because the intersection $P\cap P'$ is a line, as $P$ and $P'$ are not parallel, and the line contains exactly one point at every level, as the line does not lie at a fixed level. In $\bigO(n^2)$ time and memory, we can compute $P'\cap P_\alpha$ and $P'\cap P_\beta$ for all slanted $P'$. 
The set of slanted planes can be sorted by when they intersect $P_\alpha$ as we traverse $P_\alpha$ in one direction. This gives a total preorder $\leq_\alpha$ on the set of slanted planes with $P'\leq_\alpha P''$ if $P'$ intersects $P_\alpha$ before or at the same point as $P''$. This is not necessarily a total order, as several planes might intersect $P_\alpha$ in the same point. Define a second total preorder $\leq_\beta$ by doing the same for $P_\beta$, traversing it in the same direction as $P_\alpha$. See \cref{Fig:P_alpha_beta}.
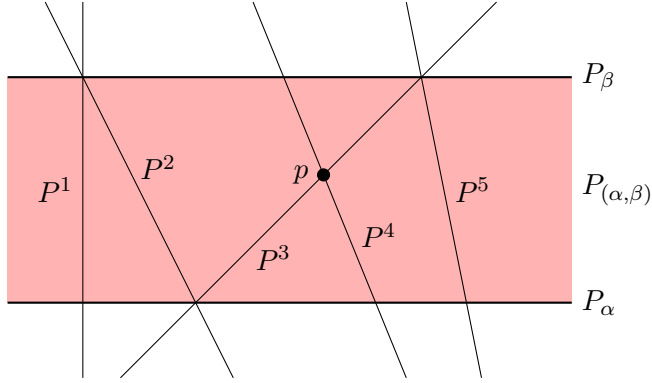
\begin{figure}
\centering
\begin{tikzpicture}
\fill[red, opacity=0.3]
(-.5,0) to (7,0) to (7,3) to (-.5,3);
\draw[thick] (-.5,0) to (7,0);
\node[right] at (7,0){$P_\alpha$};
\draw[thick] (-.5,3) to (7,3);
\node[right] at (7,3){$P_\beta$};
\node[right] at (7,1.5){$P_{(\alpha,\beta)}$};
\draw (.5,-1) to (.5,4);
\node[left] at (.5,1.5){$P^1$};
\draw (2.5,-1) to (0,4);
\node at (1.5,1.8){$P^2$};
\draw (1,-1) to (6,4);
\node[left] at (3.4,.6){$P^3$};
\draw[color=black,fill=black] (3.7,1.7) circle (.08);
\node at (3.4,1.7){$p$};
\draw (4.8,-1) to (2.763,4);
\node[left] at (4.8,.9){$P^4$};
\draw (5.8,-1) to (4.8,4);
\node[left] at (6.05,1.5){$P^5$};
\end{tikzpicture}
\caption{Illustration of $P$ with $P_\alpha$ and $P_\beta$ as well as the intersection with slanted planes $P^1,\dots,P^5$. $P_{(\alpha,\beta)}$ is the open region in pink. We have $P^1<_\alpha P^2 =_\alpha P^3 <_\alpha P^4 <_\alpha P^5$ and $P^1=_\beta P^2 <_\beta P^4 <_\beta P^3 =_\beta P^5$. The only $i$ and $j$ for which both $P^i <_\alpha P^j$ and $P^j <_\beta P^i$ hold are $i=3$ and $j=4$, and this corresponds to the only vertex in $P_{(\alpha,\beta)}$ arising from the planes we have drawn, namely $p\in P\cap P^3 \cap P^4$.}
\label{Fig:P_alpha_beta}
\end{figure}

$P\cap P'\cap P''$ is a vertex in $P_{(\alpha,\beta)}$ if and only if 
$P'<_\alpha P''$ and $P''<_\beta P'$, or vice versa.
Whether $P'$ and $P''$ satisfying these inequalities exist can be checked in $\bigO(n^2\log n)$: Fix an arbitrary total order $<$ of the slanted planes. List the (at most $n^2$) slanted planes in increasing order by $\leq_\alpha$, using $\leq_\beta$ as a tiebreaker if $P' =_\alpha P''$ and $<$ as a tiebreaker if in addition $P' =_\beta P''$ (in this case, $P\cap P'\cap P''$ is a line) in $\bigO(n^2\log n)$ time. This gives a total order on the set of slanted planes.
Then list the slanted planes by the same method with $\leq_\alpha$ and $\leq_\beta$ switching roles. There are $P'$ and $P''$ such that $P'<_\alpha P''$ and $P''<_\beta P'$ if and only if these two lists are not identical:
If $P'<_\alpha P''$ and $P''<_\beta P'$, then $P'$ and $P''$ are in opposite orders in the two lists, and if $P'\leq_\alpha P''$ and $P'\leq_\beta P''$, then they are in the same order. We can check if the lists are identical by running through them simultaneously in $\bigO(n^2)$ time and memory.
\end{proof}
\begin{lemma}
\label{Lem:compute_I_p}
If $(T(n),M(n))$ are the time- and space-complexity of an algorithm to decide if $d_\M\leq \lambda$ for a given $\lambda$, then for any $P\in\T$, we can compute $I_P$ in $\bigO((n^4+T(n))\log n)$
time and $\bigO(n^2+M(n))$ memory.
\end{lemma}
\begin{proof}
We start with a simplified version of the proof that shows the correct running time but uses too much memory:
For each pair $P'$, $P''$ of planes in $\T$, we can compute $P\cap P'\cap P''$ in constant time. If this is a single point, it is a vertex. In $\bigO(n^4\log n)$ time, we can find all the vertices of $\T$ in $P$ by running through all pairs of planes in $\T$, list their levels and sort them in increasing order. Adding $c_{K+1}=\infty$ and possibly $c_0=0$, we get a list $(c_0<)c_1<\dots<c_{K+1}$ as in \cref{Def:I_P}. Using binary search, we can find $c_i$ and $c_{i+1}$ such that $c_i<d_\M\leq c_{i+1}$ in $\bigO(T(n)\log n)$ time.
By definition, $I_P=(c_i,c_{i+1}]$.

The problem with the above argument is that it requires $\bigO(n^4)$ space to store the levels of all intersections.
To avoid that, observe that instead of doing binary search with the median element of the remaining search range, it suffices
to identify an \emph{approximate median} $x$, that is, an element such that if $k$ elements are in the search range,
at least $k/8$ elements are smaller than $x$ and at least $k/8$ elements are larger than $x$. Indeed, performing the
binary search with such approximate medians results in $\bigO(\log n)$ iterations, with a slightly larger constant
hidden in the $\bigO$ notation. We now argue that we can identify an approximate median for any given range in $\bigO(n^4)$
time and $\bigO(n^2)$ memory. With that, the theorem follows.

The idea is reminiscent of the famous ``median of medians'' algorithm to compute medians in linear time~\cite{blum73}; see also \cite[Sec.~9.3]{cormen}.
Let $I$ denote an interval that contains $d_\M$.
We maintain a container of size $n^2$, which we call the \emph{batch container}
and another container of size $n^2$, called the \emph{median container}, both initially empty.
We traverse the $\bigO(n^4)$ pairs of planes as above and compute the level of the intersection points.
If a level does not lie in $I$, we disregard it. Otherwise, we add it to the batch container.
If the batch container is full, we compute its (exact) median (using the algorithm in \cite[Sec.~9.3]{cormen}).
This requires $\bigO(n^2)$ time and space. We add this median to the median container and clear the batch container.
After the last intersection is traversed, we also compute the median of the remaining elements in the batch container
and add it to the median container. Then, we compute the median of the median container, again in $\bigO(n^2)$
time and space, and return the median $x$ as the result.

It is clear that the algorithm runs in $\bigO(n^4)$ time and uses $\bigO(n^2)$ memory. To see that the outcome $x$ is
indeed an approximate median, assume for simplicity that the number $k$ of elements in $I$ is a multiple of $n^2$, say $k=mn^2$
and assume $m=2\ell+1$ is odd. Then, there are $\ell$ medians smaller than $x$, and each such medians was selected
as median of $n^2$ values, so there are $n^2/2$ elements in its batch that are smaller and hence, also smaller than $x$.
Hence, there are $\ell n^2/2\geq k/4$ elements smaller than $x$, and the symmetric argument holds for elements larger than $x$.
Slightly more care is required if the last batch has fewer than $n^2$ elements, but a careful analysis shows
that bound of $k/8$ if $k$ is larger than a sufficiently large constant. We omit further details.
\end{proof}
\begin{theorem}
\label{Thm:runtime_dep_on_T(n)}
If $(T(n),M(n))$ be the time- and space-complexity of an algorithm for deciding if $d_\M\leq \lambda$ for a given $\lambda$, then we can compute $d_\M$ in expected time $\bigO((n^4+T(n))\log^2 n)$
and with $\bigO(n^2+M(n))$ memory.
\end{theorem}
\begin{proof}
We want to compute $(\alpha,\beta]=\bigcap_{P\in \T}I_P$: Since all $I_P=(\alpha_P,\beta_P]$ contain $d_\M$, $(\alpha,\beta]$ contains $d_\M$. Since $(\alpha_P,\beta_P)$ does not contain the level of any vertex in $P$, $(\alpha,\beta)$ does not contain the level of any vertex and thus does not contain $d_\M$. We conclude that $d_\M=\beta$.

We iterate through all planes in $\T$ in a random order $P_1, P_2,\dots, P_{|\T|}$. We first compute $I_{P_1}$ and let $I_1=I_{P_1}$. For all subsequent $P_i$, we assume $I_{i-1} = \bigcap_{j=1}^{j=i-1}I_{P_j}$ and check if $I_{i-1}\subseteq I_{P_i}$. If the answer is yes, we let $I_i=I_{i-1}$ and move on to the next iteration. If the answer is no, we compute $I_{P_i}$ and let $I_i=I_{i-1}\cap I_{P_i}$. In both cases, $I_i = \bigcap_{j=1}^{j=i}I_{P_j}$, so it is clear that $I_{|\T|}=\bigcap_{P\in \T}I_P$.

Let $N$ be the number of $i$ for which $I_{i-1}\nsubseteq I_{P_i}$. By \cref{Lem:decide_inclusion} and \cref{Lem:compute_I_p}, the runtime of this algorithm is
\begin{equation}
\label{Eq:runtime_N}
\bigO(n^2\cdot n^2\log n+N(n^4+T(n))\log n).
\end{equation}
and it uses $\bigO(n^2+M(n))$ memory.
To find the expected runtime, we need to estimate $N$. For $I_{i-1}=\bigcap_{j=1}^{j=i-1}I_{P_j}\nsubseteq I_{P_i}$ to hold, we must either have $\alpha_{P_i}>\alpha_{P_j}$ for all $j<i$, or $\beta_{P_i}<\beta_{P_j}$ for all $j<i$. The probability for the first inequality to hold is at most $\frac{1}{i}$, as in the sequence $\alpha_{P_1},\dots,\alpha_{P_i}$, there has to be a unique smallest number, this number has to be placed last, and it is equally likely to be in each of the $i$ positions. A similar argument shows that also the probability of $\beta_{P_i}<\beta_{P_j}$ is at most $\frac{1}{i}$. Putting the two together, the probability of $I_{i-1}\nsubseteq I_{P_i}$ is bounded above by $\frac{2}{i}$. Summing over all $i$ then shows that the expected value of $N$ is $\bigO(\log n)$. We get the expected runtime given in the lemma by putting this into \cref{Eq:runtime_N}.
\end{proof}
\begin{corollary}
We can compute $d_\M$ in expected time $\bigO(n^5\log^3 n)$ time and $\bigO(n^2)$ memory.
\end{corollary}
\begin{proof}
By \cref{Thm:runtime_leq_lambda}, we can take $T(n)$ to be $\bigO(n^5 \log n)$ and $M(n)=\bigO(n^2)$.
Inserting these $T(n)$ and $M(n)$ in bounds of \cref{Thm:runtime_dep_on_T(n)}, we get the theorem.
\end{proof}

\section{A more space-efficient algorithm.}
\label{sec:space-efficient}
The exact matching distance can also be computed with a linear space bound,
to the price of a much \change{worse} time bound. Since we consider this approach
of mostly theoretical interest and irrelevant in practice, we only sketch the main
ideas for the sake of brevity:

Let us first revisit the decision algorithm for a fixed $\lambda$. \change{Using} a binary counter, we can iterate through all lines of the arrangement 
from Definition~\ref{def:lambda_lines} using only logarithmic space. 
For every line $L$, we can iterate through the $O(n^2)$ intersection points with
all other lines in order without storing all intersections at once: this is done by scanning through all lines
to find the leftmost intersection with $L$, and in later steps by finding the leftmost intersection point that lies right
of the previous intersection encountered, also by a linear scan. That can be done in linear space (but requires quadratic time per point encountered).
For any two consecutive intersection points along $L$, the corresponding line segment is incident to two faces, and by moving
slightly above and below $L$, we can find points in the interior of these faces. We compute the barcodes and
their bottleneck distance at these points; if the result is greater than $\lambda$, we stop. If the algorithm iterates
through all lines of the arrangement, we can be sure that every face has been checked (in fact, every face
has been checked once for every boundary edge). Since the barcode consists of $n$ elements, and computing the bottleneck
distance requires linear space in the size of the barcode, we can state:

\begin{lemma}
Let $(T_{\mathrm{pers}},S_{\mathrm{pers}})$ be the time and space complexity for computing the barcode of a presentation
of a $1$-parameter module of size $n$. Then, we can decide whether $d_\M(M,N)\leq \lambda$
in time $O(n^c + n^d T_{\mathrm{pers}})$ for some constants $c,d$ using $O(n+S_{\mathrm{pers}})$ space.
\end{lemma}
If we compute barcodes via matrix-multiplication, we obtain $T_{\mathrm{pers}}=O(n^\omega)$ and $S_{\mathrm{pers}}=O(n^2)$
which yields a worse time complexity bound than the approach in Section~\ref{Sec:deciding} with no benefit on the memory
consumption. However, an unpublished
barcode algorithm by Ulrich Bauer\footnote{This algorithm is called ``oblivious matrix reduction''~-- see \url{https://ulrich-bauer.org/ripser-talk.pdf}, slides 71--73.} yields $S_{\mathrm{pers}}=O(n)$
to the price that the time complexity is exponential. Still, ignoring the time complexity, instantiating
our decision algorithm with Bauer's barcode algorithm yields a linear space algorithm.

To compute the exact matching distance, we can just enumerate all triples of planes from Definition~\ref{def:plane_arrangement},
again using a counter, and compute for each triple the intersection point. We query the decision algorithm for the $\lambda$-value
of the intersection point, and remember the largest lambda for which the decision algorithm gives a positive answer.
This requires $O(n^6)$ calls of the decision algorithm as opposed to $O(\log n)$ such calls as in our more time-efficient main result. 
However, the space consumption is better:

\begin{theorem}
Let $(T_{\mathrm{pers}},S_{\mathrm{pers}})$ be the time and space complexity for computing the barcode of a presentation
of a $1$-parameter module of size $n$. Then, we can compute $d_\M(M,N)$
in time $O(n^c + n^dT_{\mathrm{pers}}))$ for some constants $c,d$ using $O(n+S_{\mathrm{pers}})$ space.
\end{theorem}

A slightly more careful analysis yields that one can set $c=12$ and $d=10$ in the above theorem. Again, using Bauer's algorithm,
we obtain an algorithm with space complexity $O(n)$.

\section{Conclusion.}
Our approach considerably improves the time and space complexity
of computing the exact matching distance for the case of $2$ parameters.
An improvement in the time complexity would be possible
via a more efficient algorithm for the decision problem in Section~\ref{Sec:deciding}, 
since its time complexity is the dominating factor of the current analysis.
For instance, a $\tilde\bigO(n^4)$ time algorithm for the decision problem
would immediately yield of $\tilde\bigO(n^4)$ (expected) time algorithm for the computation
of the matching distance as well using Theorem~\ref{Thm:runtime_dep_on_T(n)}.
While we demonstrate that the space complexity can be improved in principle,
we speculate that such an improvement will be very difficult without
a serious worsening of the time complexity. It is not even clear to us
whether computing the matching distance can be done in linear space
using a polynomial-time algorithm.
The result outlined in Section~\ref{sec:space-efficient} demonstrates
that this question reduces to computing barcodes in linear space
and polynomial time.

It is natural to ask whether the algorithm has the potential to be 
used in practice, in comparison with approximate methods~\cite{kn-efficient}.
We speculate that a direct implementation of our approach might not
perform well in practice, but that a competitive
exact algorithm is possible based on the ideas. However, substantial
algorithm engineering will be needed. Some ideas, such as adaptive local
computations and the use of minimal presentations~\cite{fkr-compression,lw-computing} were already described in the conclusion of~\cite{klo-exact}.
Moreover, a major source for improvement might be to reduce the size
of the dual line arrangements considered, as not every pair of grades necessarily
needs to be considered. We leave these aspects to further work.

It is also natural to ask whether the matching distance can be computed
exactly for $3$ or more parameters.
Based on informal considerations, for $d$ parameters 
we conjecture that deciding if $d_\M\leq \lambda$ can be done by traversing a hyperplane arrangement of size $O(n^2)$ in $\R^{2d-2}$, 
and that our strategy of lifting to a hyperplane arrangement in $\R^{2d-1}$ to compute $d_\M$ exactly is viable.
We are not sure, however, to what extent our algorithm
and its complexity analysis also carries over to that case.

\end{document}